 \documentclass[11 pt]{article}
\usepackage{amssymb, amsmath, amsfonts,amsthm}
\usepackage[a4paper,vmargin={3.5cm,3.5cm},hmargin={2.5cm,2.5cm}]{geometry}
\usepackage{mathrsfs} 
\usepackage[usenames,dvipsnames,svgnames,table]{xcolor}
\usepackage{url} 
\usepackage{enumitem} 
\setlist[enumerate,1]{label={(\roman*)}}
\usepackage{bbm} 
\usepackage{lineno} 
\usepackage{ mathtools} 

\usepackage[colorlinks,
            linkcolor=Blue,
            anchorcolor=red,
            citecolor=RoyalPurple,
            ]{hyperref} 

\usepackage{titlesec} 
\titleformat{\section}[block]
{\normalfont \large\bfseries}
{\thesection}{1.2em}{\bfseries}
\titleformat{\subsection}[block]
{\normalfont \normalsize \bfseries}
{\thesubsection}{0.9em}{\bfseries}
\titlespacing{\paragraph}{%
  0pt}{
  0.5\baselineskip}{
  1em}

\usepackage{setspace}
\makeatletter
\def\th@plain{%
  \thm@notefont{}
  \itshape 
}
\def\th@definition{%
  \thm@notefont{}
  \normalfont 
}
\makeatother

\newtheorem{theorem}{Theorem}[section]

\newtheorem{lemma}[theorem]{Lemma}
\newtheorem{proposition}[theorem]{Proposition}
\newtheorem{remark}[theorem]{Remark}

\theoremstyle{definition}

\numberwithin{equation}{section}

\newcommand{\NN}{\mathbb{N}}
\newcommand{\ZZ}{\mathbb{Z}}
\newcommand{\RR}{\mathbb{R}}
\newcommand{\TT}{\mathbb{T}} 
\newcommand{\ind}[1]{\mathbbm{1}_{\left\{ #1 \right\} } } 
\newcommand{\dd}{\mathrm{d}} 

\newcommand{\EE}{\mathbb{E}} %
\newcommand{\PP}{\mathbb{P}} 

\newcommand{\Pm}{\mathrm{P}} %
\newcommand{\Ps}{\mathbf{P}} %
\newcommand{\rt}{\mathbf{o}} 

\newcommand{\Pw}{\mathbf{P}_{\omega}} 
\newcommand{\Ew}{\mathbf{E}_{\omega}} 
\newcommand{\Pwo}{\mathbf{P}_{\omega}^*} 

\newcommand{\Pws}{\mathrm{P}_{\omega}} 

\newcommand{\eps}{\varepsilon} 

\title{\textsc{Biased branching random walks\\  on Bienaymé--Galton--Watson trees}}
\author{Julien Berestycki\footnote{Department of Statistics, University of Oxford; julien.berestycki@stats.ox.ac.uk}, Nina Gantert\footnote{SoCIT, Department of Mathematics, Technical University of Munich; gantert@ma.tum.de}, David Geldbach\footnote{Department of Statistics, University of Oxford; {david.geldbach@stats.ox.ac.uk}}, 
	Quan Shi\footnote{Academy of Mathematics and Systems Science, Chinese Academy of Sciences; {quan.shi@amss.ac.cn}}}  

\date{\today}
\begin{document}
\maketitle
\begin{abstract}
We study $\lambda$-biased branching random walks on Bienaymé--Galton--Watson trees in discrete time. We consider the maximal displacement at time $n$, $\max_{\vert u \vert =n} \vert X(u) \vert$, and show that it almost surely grows at a deterministic, linear speed. 
We characterize this speed with the help of the large deviation rate function of the $\lambda$-biased random walk of a single particle.
A similar result is given for the minimal displacement at time $n$, $\min_{\vert u \vert =n} \vert X(u) \vert$.
\end{abstract}

\section{Introduction and main results}

Branching particle systems have become a staple of probability theory. On the one hand they appear naturally in random models of evolving populations, turbulence cascades, or epidemiology to name but a few of the contexts where they have proven a key tool; and on the other, in spite of their simple definition  such processes have a rich and sophisticated structure while being remarkably useful  to study other important mathematical objects such as reaction diffusion equations or log-correlated fields.

Unless there is extinction, a branching particle system can be thought of as an expanding cloud of particles in some space and the first question one wants to answer is {\it at what linear speed does this expansion happen?} This is now very well understood in homogeneous deterministic space (e.g. for branching random walks or branching Brownian motion in $\mathbb{R}^d$).  
In discrete time and space, the random walk can be replaced by any irreducible Markov chain and the corresponding branching Markov chain can be described as follows. Start the system with one particle at the origin. After each time unit, particles produce offspring according to a fixed offspring law and the offspring particles choose a new location, independently of each other, according to the Markov chain, starting from the location of the mother particle. This model is also known as {\sl tree-indexed Markov chain}. One may ask about recurrence/transience of the branching Markov chain: is the origin (or any other site in the state space of the Markov chain)  visited infinitely often by some particle? Does the maximal distance of the particles to the origin grow at a linear speed? In the transient case, does the minimal distance of the particles to the origin grow at a linear speed? These questions have also been addressed in cases where the Markov chain is a random walk in random environment, see Section \ref{Related work}.\\
\noindent
Here, we consider a branching random walk on a Bienaymé--Galton--Watson tree where the underlying Markov chain is the $\lambda$--biased random walk on a Bienaymé--Galton--Watson tree. The motion of a single particle is a well-studied model of a random walk in random environment, although there are still fascinating open questions; see
\cite{lyons_probability_2016}.
In a homogeneous environment, i.e. for a transitive Markov chain, it is a well-known result \cite[Theorem 2.12]{Mueller2008} that the origin is visited infinitely often by some particle if and only if the product of the mean offspring number with the spectral radius of the Markov chain is strictly larger than $1$. We show that the same is true in our model, see Theorem \ref{thm:ls}.\\
If the Markov chain is \emph{statistically transitive} and if the distance to the origin of the Markov chain satisfies a large deviation principle, the maximal distance of the particles to the origin grows at a linear speed which should be given as follows. The linear speed $v$ is such that the exponential growth rate of the number of particles compensates the exponential decay of the probability for a single particle to be at time $n$ at a distance at least $nv$ from the origin, see formula \eqref{velocity-descr} for a precise statement. 
We refer to Section \ref{Related work} for known examples where this heuristic has been confirmed.
Indeed, this statement is true in our model: the main results of the paper, Theorem \ref{thm:max} and Theorem \ref{thm:min}, characterize the linear growth rate of both the maximal and minimal distances of the particles to the origin. While the proof of the upper bound of the linear growth rate of the maximal distance to the origin follows a standard argument (relying on a union bound and the many--to--one formula), the proof of the lower bound is more complicated and involves several decoupling procedures.\\
\noindent
The paper is organized as follows. In Section \ref{The model}, we define the model and describe our assumptions. In Section \ref{Results}, we state our main results, Theorems \ref{thm:max} and \ref{thm:min}. We then give in Section \ref{Outline} an outline of the proof of Theorem \ref{thm:max}, which is mainly based on Proposition \ref{prop:annealed_survival_proba} and Proposition \ref{prop:0-1-law-for-liminf}. In Section \ref{Related work}, we describe some related works. Turning to the proofs, we start by recalling large deviation statements for $\lambda$-biased random walk (of a single particle) on Bienaymé--Galton--Watson trees in 
Section \ref{subsection:lambda_biased_rw}. 
In Section \ref{subsection:recurrence_and_transience}, we study the recurrence property and prove Theorem~\ref{thm:ls}. Section~\ref{sec:upper} then recalls the many-to-one formula and provides a direct application.
In Section \ref{sec:0_1_laws} we recall a well-known zero--one law and use it to prove Proposition~\ref{prop:0-1-law-for-liminf}. 
In Section \ref{sec:annealed_survival_proba} we prove Proposition~\ref{prop:annealed_survival_proba} by using a comparison with a branching process in random environment.
We then show Theorem \ref{thm:max} and Theorem \ref{thm:min}. Finally, we conclude with some open questions in Section \ref{Open}.

\subsection{The model}\label{The model}
First, we need a distribution for the environment: Let $p:=(p_k, k= 0,1,2,\ldots)$ be a probability measure on $\mathbb{N}_0 = \{0,1,2,\ldots\}$. 
Define $\mathtt{BGW}$ to be the law of a Bienaymé--Galton--Watson (BGW--tree) $\omega$ with offspring distribution $p$. 
We will always assume that
\begin{equation}\label{eq:no_leaves_bgw}
p_0=0 ~\text{and}~ p_1<1.
\end{equation}
Then the Bienaymé--Galton--Watson tree is without leaves and supercritical, i.e.\ with 
\[
m_{\mathtt{BGW}}:= \sum_{k\ge 1} k p_k >1.  
\]
Throughout this paper, we also suppose that
\[
m_{\mathtt{BGW}}<\infty. 
\]
The root of the tree $\omega$ is denoted by $\rt$. 
For  a vertex $x \in \omega$, write $|x|$ for its height, that is the (graph) distance between $x$ and $\rt$. Denote by $x_-$ its parent, by $\kappa_x$ the number of its children, and by $xj\in \omega$ its $j$-th child for $j\le \kappa_x$. 
We will denote the children of the root by $1,2, \ldots, \kappa_\rt$ instead of $\rt 1,\rt 2, \ldots, \rt\kappa_\rt$. Let $D_n$ be the family of vertices with height equal to $n$; we call $D_n$ the $n$-th level of the tree.  
For every $i\le |x|$, let $x_i \in D_i$ be its ancestor at the $i$-th level. 
 Write $x\preceq y$ if $x$ is an ancestor of $y$ or $x=y$. 

Next, we need a branching mechanism: Let $\mu:=(\mu_k, k= 0,1,2,\ldots)$ be another probability measure. Define $\TT$ to be another Bienaymé--Galton--Watson tree with offspring distribution $\mu$ independent of $\omega$. This tree will act as the genealogy of the branching random walk. To differentiate between $\TT$ and $\omega$ in notation, we call vertices of $\TT$ particles and typically denote them by $u,v,w$ (as opposed to $x,y,z\in \omega$). Similarly, we denote the root of $\TT$ by $\varnothing$, the number of children of a particle $u$ by $\gamma_u$ and we call $\vert u \vert$ the generation of $u$.

Thirdly, given the environment $\omega$ and a parameter $\lambda \geq 0$, we construct the $\lambda$--\emph{biased branching random walk} (BRW) on $\omega$ in discrete time: We denote the branching random walk by $(X(u),u\in \TT)$ and its law by $\Pw$. Unless otherwise noted, the initial particle $\varnothing \in \TT$ starts at the root $\rt \in \omega$, i.e.\ $X(\varnothing)=\rt$. Then we proceed inductively. At time $n+1$, each particle $u\in \TT$ with $\vert u \vert =n$ is replaced by a set of offspring, independently of all other particles in the $n$-th generation. The number of children 
 $\gamma_u$ follows the distribution 
\begin{equation*}
	\Pw(\gamma_u= k)=\mu_k,  \qquad k\ge 0.
\end{equation*}
For each $j\le \gamma_u$, the child $uj\in \TT$ is situated at a certain position in $\omega$, independently of $\gamma_u$ and the other children,  according to the following rule:
\begin{align*}
 	\Pw(X(uj)= i | X(u)=\rt) &=\frac{1}{\kappa_{\rt}},\quad i=1,2,\ldots, \kappa_{\rt};\\
	\Pw(X(uj) = x_- | X(u) =x) &=\frac{\lambda}{\lambda +\kappa_x}, \quad x\ne \rt;\\
 	\Pw(X(uj) = xi | X(u) =x) &=\frac{1}{\lambda +\kappa_x}, \quad x\ne \rt,~ i=1,2,\ldots, \kappa_x.
\end{align*}
We further assume that the BRW is supercritical with finite expectation, in the sense that
\begin{equation}\label{supcrit}
	m:= \sum_{k\ge 1} k \mu_k \in (1,\infty).  
\end{equation}
Under this condition, it is well-known that the BRW survives with strictly positive probability. 
In fact, for simplicity, throughout this paper we work under the stronger condition that 
\begin{equation} \label{assuonBranching}
	\mu_0=0 ~\text{and}~ \mu_1<1,
\end{equation} 
such that the BRW survives almost surely. If one assumes only \eqref{supcrit}, similar results can be easily obtained by conditioning on the event of survival.

Lastly, when we consider the randomness of the environment together with the randomness of the process, we speak of the \emph{annealed} law (compared to the \emph{quenched} law $\Pw$). The annealed law is obtained by averaging $\Pw$ over the environment,
\begin{equation*}
	\PP(\cdot):= \int \Pw(\cdot) \mathtt{BGW}(\dd \omega).
\end{equation*} 
We denote expectations with respect to $\mathtt{BGW}$ by $\EE_{\mathtt{BGW}}$, e.g.\  $\EE[\ \cdot\ ] = \EE_{\mathtt{BGW}}[\mathbf{E}_\omega[\ \cdot\ ]]$. This is to emphasise that the remaining randomness depends only on $\omega$ but not the branching random walk.

\subsection{Results}\label{Results}

The main purpose of this work is to study the asymptotic behaviour of the maximal height of the branching random walk at time $n$, that is $ \max_{|u|=n} |X(u)|$, as $n\to \infty$. 

Because at generation $n$ the branching random walk has of the order of $m^n$ particles, a first moment argument and the so-called {\it many-to-one} formula suggest that the highest particle should be at a height of order $v_{\lambda,m} n $ where $v_{\lambda,m} $ corresponds to the point where the rate function of the $\lambda$-biased random walk is equal to $\log m$.  

More precisely, let $v_{\lambda}$ be the velocity of a single $\lambda$-biased random walk (see Theorem~\ref{thm:speedRW}). 
Then, the velocity $v_{\lambda,m}$ is related to the rate function $I_{\lambda}: [v_{\lambda},1]\mapsto [0,\infty)$ of the large deviation principle for $\lambda$-biased random walks, obtained in \cite{DGPZ-ld} (c.f.\ Theorem \ref{lem:ld}). It is known that $I_\lambda$ is continuous, convex, and strictly increasing. 
We have 
\begin{equation} \label{velocity-descr}
	v_{\lambda, m}= \sup\left\{a\in [v_{\lambda},1]: I_{\lambda} (a) \le \log  m \right\} .
\end{equation}
In particular, if $v_{\lambda, m} < 1$, $v_{\lambda, m}$ is the unique solution to $I_\lambda(v_{\lambda, m}) =\log m$. 
Since $I_\lambda(v_{\lambda}) =0$, we have $v_{\lambda, m}\in (v_{\lambda},1]$. Moreover, $v_{\lambda, m}=1$ if and only if 
$m \ge e^{I_{\lambda}(1)}= (\sum_{k\ge 1} \frac{k}{k+\lambda} p_k )^{-1}$. For any fixed $\lambda$, $v_{\lambda, m}$ is strictly increasing with $m$ in $(1, e^{I_{\lambda}(1)})$. 

\begin{theorem}[Velocity of the maximal displacement] \label{thm:max}
	Let $(X(u),u\in \TT)$ be a $\lambda$-biased branching random walk with reproduction law $\mu$. 
	Suppose that $\mu_0=0$, $\mu_1<1$, $m:= \sum_{k\ge 1} k \mu_k >1$ and $m<\infty$.
	Then for $\mathtt{BGW}$--a.e.\ $\omega$, we have
	\[
	\lim_{n\to \infty} \frac{1}{n} \max_{|u|=n} |X(u)|  =  v_{\lambda, m}, \qquad \Pw-a.s., 
	\]
	where $v_{\lambda, m}$ is a constant that only depends on $\lambda$, $m$ and the 
	 tree offspring law $p$ and is given by \eqref{velocity-descr}.
\end{theorem}

Besides the maximal height of a branching random walk, we can also study the minimal height at time $n$, $\min_{\vert u \vert =n} \vert X(u) \vert$, as $n \to \infty$. The study of this quantity is related to the question of recurrence and transience of the branching random walk, or in other words its {\it local extinction/local survival}. Let $d_{\min}:= \min\{k\colon p_k >0\}$ be the minimal offspring number of $\omega$ and let 
\[
\alpha_x:= \Pw \left( \forall n\in \NN \ \exists u\in \TT: \vert u \vert\geq n, X(u)=x \middle\vert  X(\varnothing) =x \right),
\]
the quenched probability that starting at $x$, there are infinitely many particles that visit $x$. 
 
 \begin{theorem} \label{thm:ls}
	For $\mathtt{BGW}$-a.e. $\omega$, a BRW $(X(u), u\in \TT)$ is transient in the sense that $\alpha_x = 0$ for all $x\in \omega$
	if and only if 
		\begin{equation}\label{eq:rec}
 \lambda < d_{\min} \quad \text{   and   }	\quad	m \leq 
			\frac{d_{\min}+\lambda}{2\sqrt{\lambda d_{\min}}},
	\end{equation}	
	where $ m= \sum_{k\ge 1} k \mu_k> 1$ is the  mean offspring of the BRW. \\
 	Otherwise if $\lambda \geq d_{\min}$ or $m>\frac{d_{\min}+\lambda}{2\sqrt{\lambda d_{\min}}}$,  then the BRW is strongly recurrent, i.e.\ $\alpha_x = 1$ for all $x\in \omega$.
 \end{theorem}
 
We prove Theorem~\ref{thm:ls} in Section \ref{subsection:recurrence_and_transience}. We remark that part of this statement is the non--existence of the \emph{weakly recurrent} phase: there is no choice of $\lambda$, $m$, and $x$ such that we have $\alpha_x \in (0,1)$. 
The second part of condition \eqref{eq:rec} can also be written as 
$m\le e^{I_{\lambda}(0)}$; here $I_\lambda:[0,v_\lambda] \mapsto [0,\infty)$ is the large deviation rate function related to the slow down of the $\lambda$--biased random walk, see Theorem \ref{lem:ld2}. The expression for $I_\lambda(0)$ (respectively the threshold in \eqref{eq:rec}) may seem complicated. We note that $I_\lambda(0) = H(\frac{1}{2} \vert \frac{d_{\min}}{d_{\min}+\lambda})$ where $H(s\vert t)= s\log \frac{s}{t} + (1-s)\log \frac{1-s}{1-t}$ is the relative entropy between two Bernoulli distributions. Note that the critical value for $\lambda$, namely $d_{\min}$, which is required for a transient phase of the branching random walk is different from the critical value of $\lambda$ for recurrence/transience of the $\lambda$--biased random walk which is $m_{\mathtt{BGW}}$ \cite[Proposition~6.4]{Lyo90}. See Figure \ref{fig:recurrence_regimes} for an illustration of the different recurrence and transience regimes.
	
	For the case $\lambda =1$, it is proven in \cite{Su2014} that a BRW is transient in the sense that $\alpha_\rt =0$ if and only if \eqref{eq:rec} holds. The fact that recurrence or transience is determined only by $\mathrm{supp} \{p_k, k \geq 1\}$ and not by the actual distribution $ \{p_k, k \geq 1\}$ can also be compared to \cite[Theorem 1.1]{CometsPopov2007} where the authors study a BRW in random environment on $\ZZ^d$. There the recurrence/transience criterion depends only on $\mathrm{supp} \ Q$ (in their notation). 
	
\begin{figure}[tbht]
	\begin{center}
		\includegraphics[width=0.6\textwidth]{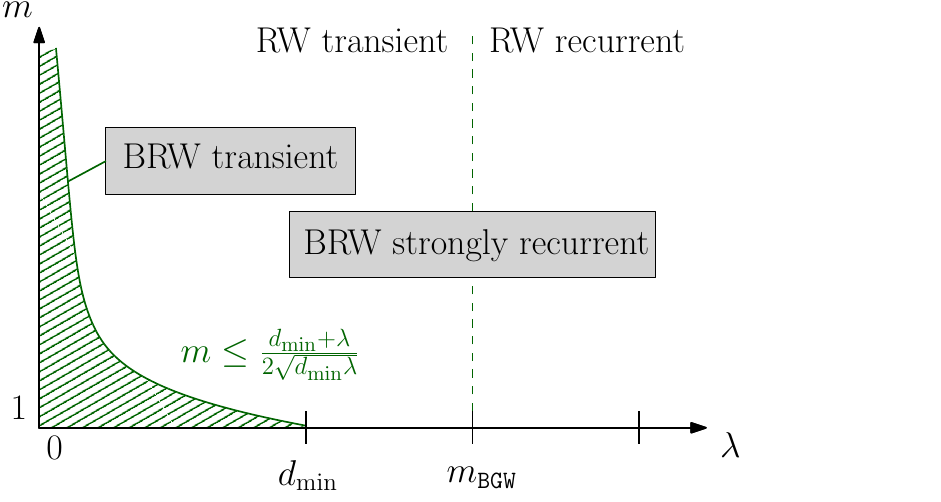}
		\caption{The different recurrence and transience regimes for the $\lambda$--biased random walk and $\lambda$--biased branching random walk.}
		\label{fig:recurrence_regimes}
	\end{center}
\end{figure}
	
Clearly, in the strongly recurrent regime we have
\begin{equation*}
	\liminf_{n\to\infty} \min_{u\in \TT, |u|=n} \frac{\vert X(u)\vert}{n} =0, \quad \Pw-a.s. \quad \text{for} \quad \mathtt{BGW}-a.e. \ \omega.
\end{equation*}
We will show that indeed
\begin{equation}\label{slower-in-rec-case}
	\lim_{n\to\infty} \min_{u\in \TT, |u|=n} \frac{\vert X(u)\vert}{n} =0, \quad \Pw-a.s. \quad \text{for} \quad \mathtt{BGW}-a.e. \ \omega,
\end{equation}
 see \eqref{slower-in-rec-case2}.
 
One might wonder about the behaviour of the number of particles at the root $\# \{u\in \TT, |u|=2n : X(u)  =\rt\}$ in this strongly recurrent phase (remember that we can only have particles at the root $\rt$ at even generations since the random walk is 2-periodic and we start from the root).  Although we refrain from proving this, we believe that in fact this number grows exponentially.

In the transient regime we show the following.

\begin{theorem}[Velocity of the minimal displacement] \label{thm:min}
	Let $(X(u), u\in \TT)$ be a $\lambda$--biased branching random walk as above. 
	Assume that we are in the transient regime, i.e.\ that $m \leq \frac{d_{\min}+\lambda}{2\sqrt{\lambda d_{\min}}}$ and $\lambda < d_{\min}$. We also need to assume 
	\begin{equation}\label{eq:assumption_for_min}
		d_{\min} \geq 2.
	\end{equation}
	Then for $\mathtt{BGW}$--a.e.\ $\omega$, we have
	\begin{equation*}
	\lim_{n\to \infty} \frac{1}{n} \min_{|u|=n} |X(u)|  =  \Tilde{v}_{\lambda, m}, \qquad \Pw-a.s., 
	\end{equation*}
	where $\Tilde{v}_{\lambda, m}\in [0,v_{\lambda})$ is a constant that only depends on $\lambda$, $m$ and the 
	tree offspring law $p$, see \eqref{velocity-descr_min} below.
\end{theorem}

The assumption \eqref{eq:assumption_for_min} comes from the slowdown large deviations result \cite[Theorem 1.2]{DGPZ-ld}; see Theorem \ref{lem:ld2} below. With the large deviation rate function $I_{\lambda}\colon [0,v_{\lambda}] \to [0,\infty)$ therein we have the identity  
\begin{equation} \label{velocity-descr_min}
	\Tilde{v}_{\lambda, m}:= \inf\left\{a\in [0, v_{\lambda}]: I_{\lambda} (a) \le \log  m \right\} .
\end{equation}
Let us comment on the properties of $\Tilde{v}_{\lambda, m}$: because $d_{\min}\geq 2$ and $\lambda < d_{\min}$, $I_\lambda$ is strictly decreasing on $[0,v_\lambda]$, continuous at $0$, and $I_\lambda(0)$ is known, see \eqref{rf-at-0} below. This implies that $\Tilde{v}_{\lambda, m}>0$ if and only if $m < \frac{d_{\min}+\lambda}{2\sqrt{\lambda d_{\min}}}$. In the critical case when $m = \frac{d_{\min}+\lambda}{2\sqrt{\lambda d_{\min}}}$, we have $\Tilde{v}_{\lambda, m} = 0$. The branching walk is still transient, which suggests that in the critical case the minimum moves at sublinear speed. Also note that $\Tilde{v}_{\lambda,m} < v_{\lambda,m}$ and that $\Tilde{v}_{\lambda,m}$ is strictly decreasing in $m$ for fixed $\lambda$.

In the case of $d_{\min} = 1$, the branching random walk can only be transient (depending on $m$) if $\lambda < 1$. However this case is not covered by the large deviations result and therefore is out of our reach for the study of the branching random walk. Nevertheless, we believe the same results to be true because the authors of \cite{DGPZ-ld} believe their large deviation result to still be true \cite[Section 7, comment 1]{DGPZ-ld}.

The proof of Theorem \ref{thm:min} is very similar to that of Theorem \ref{thm:max}. We will discuss the differences briefly in Section \ref{sec:main_thms_proofs}.

\subsection{Outline of the proof of Theorem \ref{thm:max}}\label{Outline}

The main difficulty in showing Theorem \ref{thm:max} lies in showing a lower bound, i.e.\ to show that there are particles that realise the speed $(v_{\lambda,m}-\eps)$ for some $\eps>0$ small. Our approach is not dissimilar from the standard approach (see for example \cite[Section 1.4]{shi_branching_2015}) and consists mainly of the following two propositions. 

\begin{proposition}\label{prop:annealed_survival_proba}
	Under the assumptions of Theorem \ref{thm:max} we have for any $a<v_{\lambda,m}$
	\begin{align*}
		&\PP \left( \liminf_{n\to\infty} \max_{\vert u \vert = n} \frac{\vert X(u) \vert}{n} \ge  a\right) > 0.
	\end{align*}
\end{proposition}

\begin{proposition}\label{prop:0-1-law-for-liminf}
	Under the assumptions of Theorem \ref{thm:max} we have the following $0-1$ law for any $a>0$:
	\begin{align*}
		&\Pw\left( \liminf_{n\to\infty} \max_{\vert u \vert = n} \frac{\vert X(u) \vert}{n} \geq a\right) > 0 , \quad \text{for} \quad \mathtt{BGW}-a.e. \ \omega \\
		 & \hspace{4cm} \implies \Pw\left( \liminf_{n\to\infty} \max_{\vert u \vert = n} \frac{\vert X(u) \vert}{n} \geq a\right) =1, \quad \text{for} \quad \mathtt{BGW}-a.e. \ \omega.
	\end{align*}
\end{proposition} 

We show Propositions \ref{prop:0-1-law-for-liminf} and  \ref{prop:annealed_survival_proba} in Sections \ref{sec:0_1_laws} and \ref{sec:annealed_survival_proba} respectively. 

Although Proposition \ref{prop:annealed_survival_proba}~ looks 
close to Theorem \ref{thm:max} already, there are still two more steps to do. One is to go from $>0$ to $=1$ which is provided by the $\Pw$-zero--one law, Proposition \ref{prop:0-1-law-for-liminf}. The other one is to go from the annealed law $\PP$ to the quenched law $\Pw$, for which we shall develop another zero--one law under $\mathtt{BGW}$. The main obstacle in proving Proposition \ref{prop:annealed_survival_proba} is that for two particles $u,v \in \TT$ with $\vert u \vert = \vert v \vert = n$ and $u\neq v$, the descendant populations 
\begin{equation*}
 	\{X(w);w \in \TT,u \prec w \} \quad \text{and} \quad \{X(w');w' \in \TT, v \prec w' \},
\end{equation*}
while being independent under $\Pw$, are not identically distributed because they use slightly different parts of $\omega$. When considered under $\PP$, they have the same law but are not independent anymore as they might depend on the same part of $\omega$. We overcome this by careful analysis and comparing our process to an $\NN_0$--valued branching process in random environment. 

Regarding Proposition \ref{prop:0-1-law-for-liminf}, it is not surprising that this holds as $\liminf_{n\to\infty} \max_{\vert u \vert = n} \frac{\vert X(u) \vert}{n}$ is a tail--measurable random variable for fixed environment $\omega$. Nevertheless, the standard tools like Kingman's subadditivity theorem (as used in \cite[Section 1.4]{shi_branching_2015}) or Kolmogorov's $0-1$ law cannot be applied. Again, the issue is that the increments, in any sense, are not identically distributed. 
 Interestingly, our proofs are different in the strongly recurrent and the transitive regimes. In the transient regime, the crucial ingredient of the proof is that $\omega$, while not being a transitive graph in the usual sense, is \emph{statistically transitive} in the sense that for any $x\in \omega$ the subtree $\{y\in \omega: x \preceq y\}$ has the same law as $\omega$ under $\PP$.

\begin{remark}
	Proposition \ref{prop:0-1-law-for-liminf} depends crucially on Bienaymé--Galton--Watson trees being statistically transitive. Consider a $3$--ary and a $4$--ary tree joined at a common root. While both trees are transitive, the combined tree is not. If the mean offspring $m$ for the branching random walk is close enough to $1$, Proposition \ref{prop:0-1-law-for-liminf} fails (and in fact Theorem \ref{thm:max} as well). Indeed, let $A_3$ and $A_4$ be the events that the branching random walk eventually only occupies the $3$--ary (respectively $4$--ary) part of the tree. Both these events have positive probability in the local extinction phase. On the event $A_3$ (respectively $A_4$), the velocity of the maximal displacement is that of a branching random walk on a (complete) $3$--ary tree (respectively $4$--ary). These velocities are different, hence the velocity of the maximal displacement on the joined tree is random. 
\end{remark}

\subsection{Related work}\label{Related work}

Shape theorems for branching random walks on lattices or on $\mathbb{R}^d$ are a classical topic, going back to the work of \cite{76:biggins:first, J.M.Hammersley1974, 75:kingman:first}, we also refer to
\cite{lyons_probability_2016} and to \cite{shi_branching_2015} for additional references. In the one-dimensional case, not only the linear growth of the maximal/minimal distance to the origin is known, but also the fluctuations, see \cite{aidekon_convergence_2013}. Also, there are, for the one-dimensional case, many finer results about point process convergence of the particle configuration, seen from the maximum/minimum, we again refer to \cite{shi_branching_2015} for references. We remark that in the multidimensional case, there are only few results in this direction but this is a very active field of research.

The question about the linear growth of the maximal/minimal distance to the origin for branching random walks in random (spatial) environment was investigated by \cite{Devulder2007} and \cite{ZhangHouHong2020} for one-dimensional random walk in random environment. For a general model of multi-dimensional random walk in random environments, shape theorems were proven in \cite{CometsPopov2007}. 
In the discrete setup, recurrence and transience for branching Markov chains were investigated in \cite{Benjamini-Peres-indexed},\cite{GantertMueller_critical_branching} and \cite{Mueller2008}.
For branching random walks on Galton-Watson trees, where the underlying Markov chain is a simple random walk, the question about recurrence and transience was answered in \cite{Su2014}.
There are several papers studying the Martin boundary of branching Markov chains, relating it to the properties of the underlying Markov chain, we refer to 
\cite{bertacchi_martin_2024, candellero_boundary_2023, kaimanovich_limit_2023} among others. The authors of \cite{duquesne_scaling_2022} study the range of a critical branching random walk on regular trees and are working on a similar result for random trees. 

In the physics literature, there is interest in the study of F--KPP type reaction--diffusion equations in random environment
  \cite{bianco_reaction_2013, burioni_reaction_2012, hoffman_invasion_2019}. They are linked to branching random walks by duality; in particular the front of F--KPP equation behaves similarly to the maximal displacement of the branching random walk. 
  
For background on the $\lambda$--biased random walk we refer again to \cite{lyons_probability_2016}, to \cite{shi_branching_2015} and to the survey \cite{ben_arous_biased_2016} on biased random walks in random environment and the references therein.

\section{Proofs of main results}
\subsection{Preliminaries: biased random walks on a BGW tree}\label{subsection:lambda_biased_rw}

As mentioned in the introduction, the underlying motion to the BRW is the \emph{$\lambda$--biased random walk} $(S_n)_{n\geq 0}$ on $\omega$. We describe its quenched law $\Pws$. This is a Markov chain starting from $S_0 = \rt$ with transition kernel 
\begin{align*}
	 \Pws(S_{n+1} = i | S_{n} =\rt) &=\frac{1}{\kappa_{\rt}},\quad i=1,2,\ldots, \kappa_{\rt};\\
	\Pws(S_{n+1} = x_- | S_{n} =x) &=\frac{\lambda}{\lambda +\kappa_x}, \quad x\ne \rt;\\
	 \Pws(S_{n+1} = xi | S_{n} =x) &=\frac{1}{\lambda +\kappa_x}, \quad x\ne \rt, i=1,2,\ldots, \kappa_x.
\end{align*}
Let $\PP(\cdot):= \int \Pws(\cdot) \mathtt{BGW}(\dd \omega)$ be the annealed law. 

This is one of the most well-studied models for a random walk in random environment. For an introduction to the simple random walk on random trees see \cite[Chapter 17]{lyons_probability_2016} and for a survey on biased random walks on random graphs see \cite{ben_arous_biased_2016}. One key fact about the random walk is the following. 

\begin{theorem}[Existence of speed {\cite{lyons_biased_1996,PerZei08,Lyo90}}]\label{thm:speedRW}
	 Assume that $\omega$ has no leaves, i.e.\ $p_0=0$ and that $1<m_{\mathtt{BGW}}<\infty$. Then there is $v_\lambda \in [0,\infty)$ such that 
\begin{equation*}
	\Pws\left(\lim_{n \to \infty} \frac{\vert S_n \vert}{n} = v_\lambda\right)=1, \qquad  \mathtt{BGW}-a.s.
\end{equation*}
Moreover, $v_\lambda>0$ if and only if $\lambda<m_{\mathtt{BGW}}$.
\end{theorem}

To be more precise, \cite[Theorem 3.1]{lyons_biased_1996} shows that $v_\lambda>0$ when $\lambda<m_{\mathtt{BGW}}$,  \cite[Theorem~1]{PerZei08} implies $v_\lambda=0$ when $\lambda=m_{\mathtt{BGW}}$, and, if $\lambda>m_{\mathtt{BGW}}$, the positive recurrence \cite[Page 944]{Lyo90} implies $v_\lambda=0$. 

The function $\lambda \mapsto v_\lambda$ is still not fully understood; in particular, the monotonicity is a well-known open question except for small $\lambda$, see \cite{ben_arous_lyons-pemantle-peres_2014}. Also, Aïdekon \cite{aidekon_speed_2014} has derived an explicit formula for $v_\lambda$.

Another result about the $\lambda$--biased random walk which is crucial to us is the existence of a large deviation principle. 

\begin{theorem}[Large deviations, speed-up probabilities {\cite[Theorem~1.1]{DGPZ-ld}}]\label{lem:ld}
	Let $\lambda \geq 0$ and suppose that  $m_{\mathtt{BGW}}<\infty$. Then, there exists a continuous, convex, strictly increasing function 
	$I_{\lambda}: [v_{\lambda},1]\mapsto [0,\infty)$, with 
	\[
		I_{\lambda}(v_{\lambda})= 0 \quad \text{and} \quad I_{\lambda}(1) = -\log \sum_{k\ge 1} \frac{k}{k+\lambda} p_k, 
	\] 
	satisfying, for $b>a$, $a\in (v_{\lambda},1]$, 
	\begin{equation}
 		\lim_{n\to \infty} \frac{1}{n}\log \Pws \Big( \frac{|S_n|}{n}\in [a,b) \Big)
 		= \lim_{n\to \infty}  \frac{1}{n} \log \PP \Big( \frac{|S_n|}{n}\in [a,b) \Big) = -I_{\lambda}(a), \qquad 		\mathtt{BGW}-a.s.
	\end{equation}
\end{theorem}

\begin{theorem}[Large deviations, slow-down probabilities {\cite[Theorem~1.2]{DGPZ-ld}}]\label{lem:ld2}
	Suppose that  $\lambda < m_{\mathtt{BGW}}<\infty$ and that 
	\begin{equation}\label{eq:ldc}
		\text{either}~ d_{\min}\ge 2 ~\text{or}~ \lambda\ge 1. 
	\end{equation}
	Then, there exists a continuous, convex, nonncreasing function 
	$I_{\lambda}: [0,v_{\lambda}]\mapsto [0,\infty)$, with $I_{\lambda}(v_{\lambda})= 0$, such that for $0\le
 	b<a<v_{\lambda}$, 
	\begin{equation}
		\lim_{n\to \infty} \frac{1}{n}\log \Pws \Big( \frac{|S_n|}{n}\in [b,a) \Big)
		= \lim_{n\to \infty}  \frac{1}{n} \log \PP \Big( \frac{|S_n|}{n}\in [b,a) \Big) = -I_{\lambda}(a), \qquad 		\mathtt{BGW}-a.s.
	\end{equation}
	If $\lambda\ge d_{\min}$, then $I_{\lambda} (a)=0$ for $a\in [0,v_{\lambda}]$. 
	If $\lambda< d_{\min}$, then $I_{\lambda}$ is strictly decreasing on $[0,v_{\lambda}]$ and 
\begin{equation}\label{rf-at-0}
I_{\lambda}(0) = \lim_{a\downarrow 0} I_{\lambda}(a)= \log \Big(\frac{d_{\min}+\lambda}
	{2\sqrt{\lambda d_{\min}}}\Big) .
\end{equation}
\end{theorem}

\subsection{Strong recurrence and transience, proof of Theorem~\ref{thm:ls}}\label{subsection:recurrence_and_transience}

The recurrence property of a BRW is closely related to the \emph{spectral radius} of the $\lambda$-biased random walk $(S_n)$ of a single particle on the tree $\omega$. This is the following number, which is known to be the same for every $x,y\in \omega$:  
\begin{equation}\label{spectral-radius-def}
\rho_{\omega}(\lambda):=\limsup_{n\to \infty} \Pws(S_n= y \mid S_0= x)^{\frac{1}{n}}. 
\end{equation}

\begin{proposition}\label{prop:sr}
	Suppose that  $m_{\mathtt{BGW}}<\infty$. 
	Then  for $\mathtt{BGW}$-a.e.\ $\omega$,  
		\begin{equation}
		\rho_{\omega}(\lambda)=
		\begin{cases}
			1,& \lambda\ge d_{\min}; \\
			\frac{2\sqrt{\lambda d_{\min}}}{d_{\min}+\lambda} , & \lambda\in (0,d_{\min}). 				
		\end{cases}
	\end{equation}
\end{proposition}
For $\lambda=1$, this result can also be found in \cite[Proposition~3.5]{Su2014}. 

\begin{proof}
	If the $\mathtt{BGW}$ tree has a.s.\ bounded degree, then it follows from \cite[Lemma~2.8]{Mueller2008} that 
	$\rho_{\omega}(\lambda)  = e^{-I_\lambda(0)}$, with $I_\lambda(0)$ as in Theorem~\ref{lem:ld2}, which leads to the desired statement under the additional assumption \eqref{eq:ldc}. 
	
	In general, without these further assumptions, we provide a direct proof, using ideas from the proof of \cite[Lemma~2.8]{Mueller2008}. 
	For $s,t\in (0,1)$ let $H(s\vert t)= s\log \frac{s}{t} + (1-s)\log \frac{1-s}{1-t}$, the relative entropy between two Bernoulli distributions.  For $\lambda\in (0,d_{min})$, with $p = \frac{d_{\min}}{d_{\min}+\lambda}$, we have $H\left( \frac{1}{2} \middle\vert p \right) = -\log \left(\frac{2\sqrt{\lambda d_{\min}}}{d_{\min}+\lambda}\right) $.
	For $\mathtt{BGW}$-a.e.\ $\omega$, we have $\Pws(S_n= \rt)\le \Pws(|S_n| \le n \varepsilon)$ for each $\varepsilon>0$. Letting $\eps \to 0$, we have by definition \eqref{spectral-radius-def} that 
	\[
	\log \rho_{\omega}(\lambda)\le \limsup_{\eps \to 0}\limsup_{n\to \infty} \frac{1}{n} \log \Pws(S_n= \rt) \le \limsup_{\eps \to 0}\limsup_{n\to \infty} \frac{1}{n} \log\Pws(|S_n| \le n \varepsilon) \le -H\left( \frac{1}{2} \middle\vert p \right),
	\]
	where the last inequality follows from \cite[Lemma 2.1 and (3.5)]{DGPZ-ld}. 
	This gives the upper bound $\rho_{\omega}(\lambda)\le \frac{2\sqrt{\lambda d_{\min}}}{d_{\min}+\lambda}$ when $\lambda<d_{\min}$. 
	When $\lambda\ge d_{\min}$, we trivially have $\rho_{\omega}(\lambda) \leq 1$ because any spectral radius is bounded by $1$. This yields the desired upper bound. 
	
	To obtain the lower bound for $\rho_{\omega}(\lambda)$, we treat the two cases $\lambda<d_{\min}$ and  $\lambda \ge d_{\min}$ separately.
	 We first consider $\lambda<d_{\min}$. 
	 Let $L\in \NN$ and $x\in \omega$. Denote by $A(x,2L)$ the first $2L$ levels of the fringe tree rooted at $x$, i.e.\ $A(x,2L)= \{y\in \omega\colon x\preceq y, |y|\le |x|+2L \}$. Call $x$ a $(d_{\min},2L)$--\emph{nice} vertex if $A(x,2L)$ is a $(d_{\min})$--ary tree.

\begin{figure}[tbh]
		\begin{center}
		\includegraphics[width=0.7\textwidth]{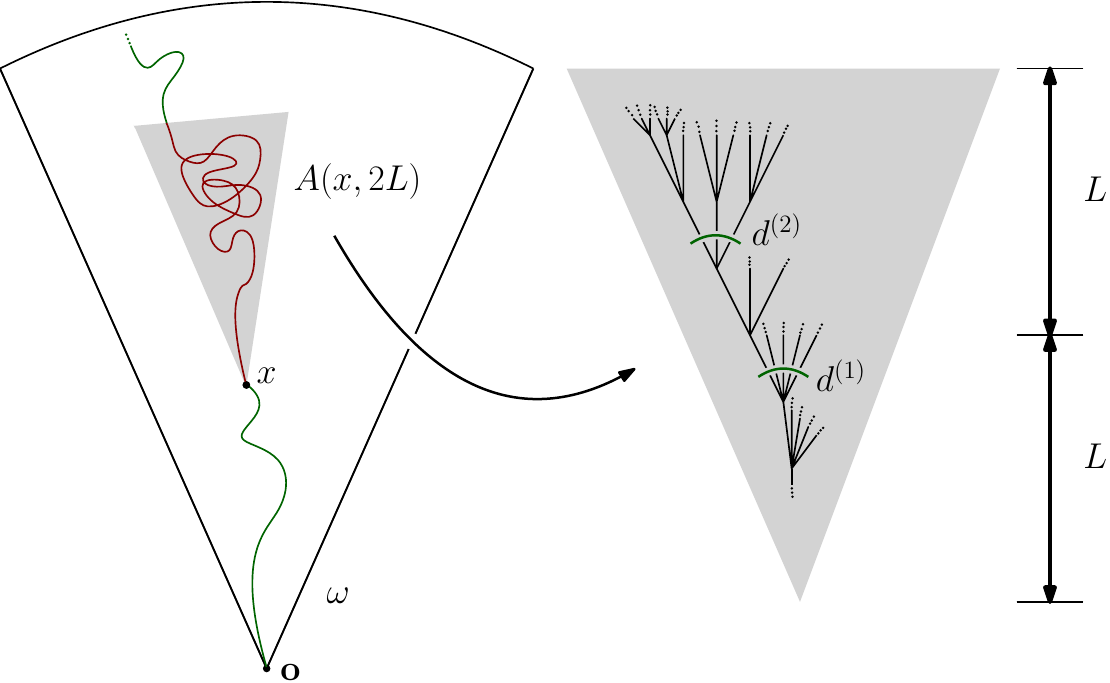}
		\caption{The spectral radius $\rho_\omega(\lambda)$ is determined by atypical regions of $\omega$: the set $A(x,2L)$ consists of $L$ levels of a $d^{(1)}$--ary tree and $L$ levels of a $d^{(2)}$--ary tree. The choice of $(d^{(1)},d^{(2)})$ depends on $\lambda$: if $\lambda < d_{\min}$ we choose $d^{(1)}=d^{(2)}=d_{\min}$ and if $d_{\min} \leq \lambda < m_{\mathtt{BGW}}$ we choose $d^{(1)}=d_0$ and $d^{(2)}=d_{\min}$ where $d_0>m_{\mathtt{BGW}}$. Then the set $A(x,2L)$ acts as a trap for the random walk. Similarly, for the branching random walk, $A(x,2L)$ acts as a seed which facilitates local survival.}
		\label{fig:spectral_radius}
		\end{center}
\end{figure}
	 
	 Such $x$ exists for $\mathtt{BGW}$-a.e.\ $\omega$ and for simplicity we write $A:= A(x,2L)$. Consider a $(d_{\min},2L)$--nice vertex and let $\mathrm{P}^{A}_{\omega}$ be the law of a $\lambda$-biased random walk killed on leaving the subgraph $A\subseteq \omega$. Choose $y_0 \in A$ arbitrarily and denote the spectral radius in $A$ by $\rho^A_{\omega}(\lambda):=\limsup_{n\to \infty} \Pws^{A}(X_n= y_0\mid X_0 = y_0)^{\frac{1}{n}}$.  
	  Naturally we have $\Pws^{A}(X_n= y_0\mid X_0 = y_0)\le \Pws(X_n= y_0\mid X_0 = y_0)$ and therefore $\rho_\omega^A(\lambda) \leq \rho_\omega(\lambda)$. Therefore it suffices to study $\rho^A_{\omega}(\lambda)$ in order to obtain a lower bound.
	 
	 Let $y_0\in A$ so that $|y_0| = |x|+ L$. Consider a random walk $(S_n)$ starting from $S_0 = y_0$ and denote  by $T$ the first time that the random walk reaches the boundary of $A$, i.e.\ $|S_n| = |x|$ or $|S_n|= |x|+2L$. Because $A$ is a $(d_{\min})$--ary tree, we can compare $ \Pws^{A}$ to a biased simple random walk on $\ZZ$. Let $\mathrm{P}_{\mathrm{SRW}(p)}$ be the law of a simple random walk on $\ZZ$ that moves up with probability $p$. Then the law of $(\vert S_n \vert, n\geq 1)$ restricted to $[\vert x \vert, \vert x \vert +2L)$ is $\mathrm{P}_{\mathrm{SRW}(p)}$ with $p = \frac{d_{\min}}{d_{\min}+\lambda}$. 
	 We imitate the $\mathrm{SRW}$--computations from \cite[Proof of (1.5), p. 249]{DGPZ-ld}. We can compare the biased $\mathrm{SRW}$ to the symmetric $\mathrm{SRW}$ by performing a change of measure. In particular we have for the large deviations of $T$ which is now the first time of the $\mathrm{SRW}$ to reach the boundary of $[\vert x \vert, \vert x \vert +2L]$ that 
	 \begin{equation*}
	 	\liminf_{n\to \infty}\frac{1}{n}\log \mathrm{P}_{\mathrm{SRW}(p)} \left( T> n \right) \geq 			\liminf_{n\to \infty}\frac{1}{n}\log \mathrm{P}_{\mathrm{SRW}(\frac{1}{2})} \left( T> n \right)-H\left( \frac{1}{2} \middle\vert p \right).
	 \end{equation*}
	 By  the large deviations of $T$ for $\mathrm{P}_{\mathrm{SRW}(\frac{1}{2})}$, see e.g.\ \cite[Page 244]{spitzer_RW}, we have
	 \begin{equation*}
	 	\liminf_{n\to \infty}\frac{1}{n}\log \mathrm{P}_{\mathrm{SRW}(\frac{1}{2})} \left( T> n \right) = - \frac{\pi^2}{8L^2}.
	 \end{equation*}
	  Thus 
	  \begin{equation*}
	 	\liminf_{n\to \infty}\frac{1}{n}\log \mathrm{P}_{\mathrm{SRW}(p)} \left( T> n \right) \geq 			- \frac{\pi^2}{8L^2} - H\left( \frac{1}{2} \middle\vert p \right) .
	 \end{equation*}
	 We return to the $\lambda$--biased random walk on $\omega$.
	 Let $c = \frac{1\wedge \lambda}{\lambda +d_{\min}}$ and $\varepsilon>0$, then for an even number $n$ with $\varepsilon n>L$, using the Markov property at time $\lceil (1-\varepsilon)n \rceil$ we have 
	 \[
	 \Pws^{A}(S_n =y_0  \mid S_0 = y_0)\\ 
	 \ge \Pws^{A}(T > \lceil (1-\varepsilon)n \rceil  \mid S_0 = y_0) c^{ \varepsilon n}.
	 \]
	 It follows that 
	 \begin{align*}
	 	\limsup_{n\to \infty} n^{-1} \log \Pws^{A}(S_n =y_0  \mid S_0 = y_0) 
	 	\ge (1-\varepsilon) \bigg(-\frac{\pi^2}{8 L^2} - H\left( \frac{1}{2} \middle\vert p \right) \bigg) + \varepsilon \log c.
	 \end{align*}
 As $\varepsilon$ is arbitrary, this yields $\rho_{\omega}(\lambda)  \ge \rho^A_{\omega}(\lambda) \ge \exp( -\frac{\pi^2}{8L^2} - H\left( \frac{1}{2} \middle\vert p \right) )$. 
 Letting $L\to \infty$, we conclude that $\rho_{\omega}(\lambda)  \ge e^{- H\left( \frac{1}{2} \middle\vert p \right)}$. 
 Recall that $p = \frac{d_{\min}}{d_{\min}+\lambda}$ and  $e^{- H\left( \frac{1}{2} \middle\vert p \right)}=  \frac{2\sqrt{\lambda d_{\min}}}{d_{\min}+\lambda}$. So we have the desired lower bound. 
	
	We next turn to the case $\lambda \ge d_{\min}$. Note that when $\lambda \ge m_{\mathtt{BGW}}$ the $\lambda$-biased random walk is recurrent - hence the spectral radius is $1$. It remains to study the case $m_{\mathtt{BGW}}> \lambda \ge d_{\min}$. The arguments are very similar as for the case $\lambda < d_{\min}$ but we consider a different subset of $\omega$, using ideas from \cite[Proof of Theorem~1.2]{DGPZ-ld}. 
	Recall that $(p_k,k\geq 1)$ is the probability distribution that determines $\mathtt{BGW}$. 
	There is $d_0 \in \mathbb{N}$ such that $p_{d_0} > 0$ with $d_0 \geq m_{\mathtt{BGW}} >\lambda$, fix such $d_0$. 
	Next, we again consider a vertex $x\in \omega$ where $A(x,2L)$ takes a specific form: on the first $L$ levels $A(x,2L)$ is a $d_0$--ary tree, whereas on the next $L$ levels $A(x,2L)$ is a $(d_{\min})$--ary tree. More precisely, for $y\succeq x$ we require $\kappa_y=d_0$ if $\vert y \vert <\vert x \vert + L$ and $\kappa_y=d_{\min}$ if $\vert x\vert +L \leq \vert y \vert < \vert x \vert +2L$. Such a vertex $x$ exists $\mathtt{BGW}$--almost surely, fix $x$ and abbreviate $A= A(x,2L)$.
	Starting the $\lambda$-biased random walk at any vertex $y_0\in A$ with $|y_0|= |x|+L$, we can again compare it to a random walk on $\ZZ$ killed upon leaving $[-L, L]$. This is an asymmetric random walk on $\ZZ$ with bias $\frac{d_0}{\lambda + d_0}>\frac{1}{2}$ on $\ZZ_-$ and bias $\frac{d_{\min}}{\lambda + d_{\min}} \leq \frac{1}{2}$ on $\ZZ_+$. As this random walk on $\ZZ$ is recurrent (when disregarding the killing), we deduce that  
	 $\lim_{L\to \infty} \rho^A_{\omega}(\lambda) = 1$. This also implies that $\rho_\omega(\lambda) =1$. 
\end{proof}

\begin{proof}[Proof of Theorem~\ref{thm:ls}]
	As $\lambda>0$, the $\lambda$-biased RW is irreducible. By \cite[Theorem~2.12]{Mueller2008}, a BRW starting from $x\in \omega$ is recurrent (strongly or weakly), if and only if $m >1/\rho_{\omega}(\lambda)$.
	See also \cite[Theorem~3.2]{GantertMueller_critical_branching}.
	Proposition~\ref{prop:sr} provides the value of $\rho_\omega(\lambda)$. It remains to show strong recurrence 
	when $m >1/\rho_{\omega}(\lambda)$. We only give details for the case $\lambda<d_{\min}$ as only a slight modification is needed for the case $m_{\mathtt{BGW}} > \lambda \ge d_{\min}$.

	Fix an integer $L>0$ (to be chosen later). As in the proof of Proposition~\ref{prop:sr}, let $x\in \omega$ be a $(d_{\min},2 L)$-nice vertex. 
	Let $\mathbf{P}^{A}_{\omega}$ denote the law of a BRW killed upon leaving  $A:= A (x, 2 L)$. 
	We have shown in the proof of Proposition~\ref{prop:sr} that $\lim_{L \to \infty} \rho^A_{\omega}(\lambda) = \rho_{\omega}(\lambda)$ and therefore for all $L$ large enough we have $m > 1/\rho^A_{\omega}(\lambda)$, fix such $L$. 
	By the definition of the spectral radius, for infinitely many $k$ the random walk $(S_n)$ satisfies $\Pws^A(S_k= x \mid S_0 = x) > m^{-k}$. Fix such a $k$ and
	start a BRW with law $\mathbf{P}^{A}_{\omega}$ at $x$. The expected number of particles after time $k$ at $x$ is bigger than $1$. Therefore, by considering the BRW at times $(k, 2k, 3k,\ldots)$ and only keeping the particles located at $x$, this yields a supercritical branching process, which has strictly positive survival probability $c>0$. Note that $c$ is the same for all $(d_{\min},2 L)$-nice vertices. 
	
	 Now let us consider a BRW starting with a particle at $\rt$. We look at the random walk $(S^*_n, n\ge 0)$ given by the ancestral lineage $(S^*_0 = X (\varnothing), S^*_1 = X(1), S^*_2 = X(11), \ldots)$. This is a $\lambda$--biased random walk.
	 Then the trace of this random walk a.s.\ encounters $(d_{\min},2L)$--nice vertices infinitely often. 
	 Indeed, let us explore the $\mathtt{BGW}$-tree as the random walk proceeds: let $(Z_i,i\geq 1)$ be the first vertices encountered by the random walk on level $ 2Li$. Note that the sets $(A(Z_i,2L),i\geq 1)$ are disjoint. For $i\in \NN$ let $\zeta_i =  \mathbbm{1}\{Z_i \text{ is a $(d_{\min},2L)$--nice vertex}\}$, the indicator that $Z_i$ is a $(d_{\min},2L)$--nice vertex. Consider now $(\zeta_i,i\geq 1)$, this sequence of random variables is \emph{i.i.d.} under $\PP$. And furthermore, $\PP(\zeta_1 = 1)>0$, i.e. there is a positive probability that $Z_1$ is a $(d_{\min},2L)$--nice vertex. This implies that $\PP$--almost surely, and hence also $\Pw$--almost surely, $\zeta_i=1$ infinitely often. 
	 
Having established that $(S_n^*, n\geq 1)$ encounters infinitely many $(d_{\min},2L)$--nice vertices, let $(Z_i^*,i\geq 1)$ be a subsequence of $(Z_i,i\geq1)$ such that for every $i\in \NN$ $Z_i^*$ is a $(d_{\min},2L)$--nice vertex. For each $i$, consider now a branching process with the measure $\mathbf{P}^{A}_{\omega}$ where $A= A(Z_i^*,2L)$ starting at $Z_i^*$ at the time when $(S_n^*,n\geq 1)$ first encounters $Z_i^*$. This auxiliary process is realised as a subset of particles of our BRW. As discussed above, this auxiliary branching process survives with probability $c>0$ under $\Pw$ for every $i$. Further, the auxiliary processes are independent. As survival of these processes are \emph{i.i.d.}\ trials, $\Pw$--almost surely, infinitely many of them survive. In particular, there exists $i$ such that the $i$--th auxiliary process survives. This means that $Z_i^*$ is visited infinitely often and hence the branching random walk is strongly recurrent. 
\end{proof}

\subsection{The upper bound}\label{sec:upper}
We prove the upper bound part for Theorem~\ref{thm:max}, that is, under the same assumptions as in Theorem~\ref{thm:max}, we have, for $\mathtt{BGW}$--almost every $\omega$, 
\begin{equation}\label{eq:conclusion_upper_bound}
	\limsup_{n \to \infty} \max_{|u|=n} \frac{\vert X(u) \vert }{n} \leq  v_{\lambda,m}, \qquad \Pw-a.s.
\end{equation}
Let us first mention a very useful tool that connects the BRW and the random walk of a single particle, the well-known \emph{many--to--one formula}. For fixed $\omega$ and any $f\geq0$ that is a function of a particle trajectory we have
\begin{equation}\label{eq:many_to_one}
	\mathbf{E}_{\omega}\left[\sum_{u\in \mathbb{T}\colon\vert u \vert = n} f\left((X(u_k), k\leq n) \right)\right] = m^n \mathrm{E}_{\omega} \left[ f\left((S_k, k\leq n) \right) \right], \qquad n\in \NN.
\end{equation}
In our case, \eqref{eq:many_to_one} follows from the linearity of expectation.
See also for example \cite[Theorem 1.1]{shi_branching_2015} for a proof for a more general model of branching random walks on $\RR$.
Analogous formulae hold for the annealed laws.

Using the many-to-one formula to translate the problem to the related RW, the upper bound follows almost immediately from the large deviation estimate in Theorem \ref{lem:ld}.
When $ v_{\lambda,m}=1$, the upper bound \eqref{eq:conclusion_upper_bound} holds automatically. So we may assume $ v_{\lambda,m}<1$ and let $a>v_{\lambda,m}$. By the many--to--one formula \eqref{eq:many_to_one} we have 
\begin{equation*}
	\Pw\left( \max_{|u|=n} \frac{\vert X(u) \vert }{n}\ge  a \right) \le  \Ew \bigg[\sum_{|u|=n} \ind{|X(u)|\ge a n} \bigg] =  m^n \Pws\left( |S_n|\ge a n \right).    
\end{equation*}
Since $a>v_{\lambda, m}$, we have $I_{\lambda}(a)>\log m$ and it follows from Theorem \ref{lem:ld} that, for $\mathtt{BGW}$-a.e.\ $\omega$, 
\[
\limsup_{n\to \infty} \frac{1}{n} \log \Pw\bigg( \max_{|u|=n} \frac{\vert X(u) \vert }{n}\ge  a \bigg) < 0.
\]
The Borel-Cantelli lemma implies that
\begin{equation*}
	\limsup_{n \to \infty} \max_{|u|=n} \frac{\vert X(u) \vert }{n} \leq  a, \qquad \Pw-a.s.,
\end{equation*}
for $\mathtt{BGW}$--almost every $\omega$. Because $a>v_{\lambda,m}$ was arbitrary, we deduce \eqref{eq:conclusion_upper_bound}.

\subsection{$\mathbf{0-1}$ laws, proof of Proposition \ref{prop:0-1-law-for-liminf}}\label{sec:0_1_laws}

The goal of this section is to show Proposition \ref{prop:0-1-law-for-liminf}. Interestingly, this requires different proofs for the strongly recurrent and transient regimes. We do this by showing various $0-1$ laws for the environment $\mathtt{BGW}$ and for the $\lambda$--biased random walk. Let us start by recalling the following classical zero-one law (c.f.~\cite[Proposition 5.6]{lyons_probability_2016}). We slightly vary notation to stress that this lemma can be applied both to the law of $\omega$ and to the law of $\TT$. 
 
\begin{lemma}\label{lem:inherited}[Zero-one law for inherited properties]\\
 Consider a Bienaymé--Galton--Watson reproduction law satisfying \eqref{eq:no_leaves_bgw} (no leaves) and denote the 
 corresponding branching process by $(Z_n)_{n\geq 0}$.
A property $A$ of trees, is called an inherited property if 
$A \subseteq \{\omega: \forall j\leq Z_1: \omega_j \in A\}$ where $\omega_j$ denotes the subtree rooted at $j$, for $j = 1, ...,Z_1$. This means that if a tree satisfies the property $A$ then all subtrees rooted at the children of the root also satisfy the property $A$. Then, if $A$ is an inherited property, we have $\mathtt{BGW}(A) \in \{0,1\}$.
\end{lemma}
\begin{proof}
Let $\mathcal{A}$ (resp.\ $\mathcal{A}^{(j)}$) be the collection of trees $\omega$ such that  $\omega$ (resp.\ $\omega_j$) satisfies property $A$.
Conditional on $Z_1$, $(\mathcal{A}^{(j)},j\leq Z_1)$ are independent copies of $\mathcal{A}$. Therefore
\[
\alpha: = \mathtt{BGW}(\mathcal{A}) \leq \mathtt{BGW}\bigg(\bigcap_{j\leq Z_1} \mathcal{A}^{(j)} \;\bigg)= \EE_\mathtt{BGW}\bigg[ \prod_{j=1}^{Z_1}\mathtt{BGW}( \mathcal{A}^{(j)})\bigg]
= \EE_\mathtt{BGW}[\alpha^{Z_1}],
\]
hence, $\alpha\leq \varphi(\alpha)$ where $\varphi(s): = \EE_\mathtt{BGW}[s^{Z_1}]$. But, due to \eqref{eq:no_leaves_bgw}, the function $\varphi$ is strictly convex with $\varphi(0) =0$ and $\varphi(1) =1$. We conclude that $\alpha \in \{0,1\}$.
\end{proof}
We now turn to Proposition \ref{prop:0-1-law-for-liminf} for the strongly recurrent regime for $(\lambda, m)$ as in Theorem \ref{thm:ls}. Here we have a proper $0-1$ law. We formulate the statements for a fixed tree $\omega$, if the tree is chosen randomly according to $\mathtt{BGW}$ then the statement is true for $\mathtt{BGW}$--almost every $\omega$.

 \begin{proposition}\label{prop:rec}
 	Assume that $\omega$ is any tree such that the $\lambda$--biased BRW is strongly recurrent.  
 	Then for every $a>0$
	\begin{equation*}
		\Pw \Big(	\liminf_{n\to \infty} \max_{\vert u \vert = n}\frac{\vert X(u) \vert}{n}   \ge a\Big)\in \{0,1\}.
	\end{equation*}
 \end{proposition}

\begin{proof}[Proof of Proposition \ref{prop:0-1-law-for-liminf} for the strongly recurrent regime.] If $(\lambda, m)$ is chosen such that the $\lambda$--biased random walk is strongly recurrent for $\mathtt{BGW}$--almost every $\omega$ according to Theorem \ref{thm:ls}, then Proposition \ref{prop:0-1-law-for-liminf} follows immediately from Proposition \ref{prop:rec}.
\end{proof}

Before proving Proposition \ref{prop:rec}, we state one intermediate lemma. 

 \begin{lemma}\label{lem:rec}
 	Assume that $\omega$ is any tree such that the $\lambda$--biased BRW is strongly recurrent. 
 	Then for every $a>0$ and $x\in \omega$, 
 	set 
 	\[
 	q_{a,x}	:= \Pw \Big(	\liminf_{n\to \infty} \max_{\vert u \vert = n}\frac{\vert X(u) \vert}{n} \ge a \,\Big|\, X_{\varnothing} = x \Big).
 	\]
 	Then we have $q_{a,x} = q_{a,\rt}$. 
 \end{lemma}

 \begin{proof}
 	Set $T_{x,\rt} = \inf\{n\ge 0\colon \, \exists u \text{ with } |u|=n, X(u) =\rt\}$, which is a.s.\ finite thanks to the 	
	strong recurrence. Let $v$ be a particle such that $|v| = T_{x, \rt}$ and $X(v) = \rt$. 
 	Set $X^{(v)}(u) = X(vu), vu\in \TT$, i.e.\ the subpopulation descended from $v$. 
 	As $\max_{\vert u \vert=n}\vert X(u) \vert \geq \max_{|u| = n - T_{x,\rt}} X(vu)$, we have 
 	\[
 		\liminf_{n\to \infty} \max_{\vert u \vert=n} \frac{\vert X(u) \vert }{n}  \ge 	\liminf_{n\to \infty}  \max_{|u| = n - T_{x,\rt}}  \frac{\vert X^{(v)}(u) \vert }{n}
 	= 	\liminf_{n\to \infty} \max_{|u| = n}  \frac{\vert X^{(v)}(u) \vert }{n}.
 	\]
 	By the branching property, $X^{(v)}$ is a BRW started from $\rt$. Therefore, 
 	\[
 		q_{a,x} \ge \Pw \left(	\liminf_{n\to \infty}\ \max_{|u| = n}  \frac{\vert X(u) \vert }{n} \ge a \,\Big|\, X(\varnothing) = \rt \right)
 	= q_{a, \rt}. 
 	\]
 	A similar argument shows that $q_{a,\rt} \ge q_{a,x}$. 
 \end{proof}

\begin{proof}[Proof of Proposition~\ref{prop:rec}]
	Let $A_a = \{	\liminf_{n\to \infty} \max_{|u|=n} |X(u)|/n < a\}$ and \\
 	$A_{a}^{(j)} = \{	\liminf_{n\to \infty} \max_{|u|=n} |X^{(j)}(u)|/n < a\}$. We can see that $A_{a}$ is an inherited property, this means that
 	$A_{a}  \subseteq \bigcap_{j\le \gamma_{\varnothing}} A_{a}^{(j)}$. Moreover, the
	subpopulations $(X^{(j)}(u) := X(ju))$ descended from each particle $j \le \gamma_{\varnothing}$ are independent under 
	the law $\Pw$ given $(X(j), j\le \gamma_{\varnothing})$. Unfortunately we cannot apply Lemma \ref{lem:inherited}
	as the $A_a^{(j)}$ do not have the same law as they all still depend on the whole of $\omega$, not just their
	respective subtree. Nevertheless, using Lemma~\ref{lem:rec}, we have
	\begin{equation*}
		1- q_{a, \rt} = \Pw(A_a) \le \Ew\bigg[\Pw \bigg(\bigcap_{j\le \gamma_{\varnothing}} A_{a}^{(j)} \;\bigg|\; X(j), j\le \gamma_{\varnothing} \bigg) \bigg]
 	 	= \Ew\bigg[ \prod_{j=1}^{\gamma_{\varnothing}}\left( 1- q_{a,X(j)} \right)\bigg]
 	 	 = \Ew[\left(1- q_{a, \rt}\right)^{\gamma_{\varnothing}}]. 
	\end{equation*}
 	 This implies that $1-q_{a, \rt}$ is equal to either $0$ or $1$, and the same applies to  $q_{a, \rt}$.
\end{proof}

Very similar arguments allow us now to complete the proof of  \eqref{slower-in-rec-case}.
\begin{proof}[Proof of \eqref{slower-in-rec-case}]
	For any $a>0$, we see that $\{	\limsup_{n\to \infty} \min_{|u|=n} |X(u)|/n > a\}$ is an inherited property.  Since we are in the strongly recurrent case, analogous result to Lemma~\ref{lem:rec} holds. By similar arguments as in the proof of Proposition~\ref{prop:rec}, we deduce the zero--one law that, for $\mathtt{BGW}$--almost every $\omega$, 
	$\Pw\left(	\limsup_{n\to \infty} \min_{|u|=n} |X(u)|/n > a\right)\in \{0,1\}$. 
	
	Fix $k \in \NN$, $k$ even and fix $\omega$. Consider all particles which are at $\rt$ at times $0, k, 2k, 3k, \ldots$, interpreting the particles at $\rt$ at time $(\ell + 1)k$ as children of a particle at $\rt$ at time $\ell k$. They form a branching process with mean offspring $m^k \Pws(S_k= \rt \mid S_0= \rt)$.
	In the strongly recurrent case, $m > \frac{1}{\rho_\omega(\lambda)}$ and \eqref{spectral-radius-def} gives us that for $\mathtt{BGW}$--almost every $\omega$ there is some $k$ such that 
	$m^k \Pws(S_k= \rt \mid S_0= \rt) > 1$. Hence, this branching process has a positive survival probability and we conclude that for $\mathtt{BGW}$--almost every $\omega$,
	\begin{equation*}
		\Pw\left(	\limsup_{n\to\infty} \min_{u\in \TT, |u|=n} \frac{\vert X(u)\vert}{n} =0\right) > 0.
	\end{equation*}
Then the zero--one law implies that $\Pw\left(	\limsup_{n\to \infty} \min_{|u|=n} |X(u)|/n > a\right)=0$ for every $a>0$.  Hence,
\begin{equation}\label{slower-in-rec-case2}
\Pw\left(	\limsup_{n\to\infty} \min_{u\in \TT, |u|=n} \frac{\vert X(u)\vert}{n} =0\right) =1.
\end{equation}
This implies \eqref{slower-in-rec-case}. 
\end{proof}

We now turn to the proof of Proposition \ref{prop:0-1-law-for-liminf} in the transient regime. The lack of strong recurrence means it is a priori difficult to compare two branching random walks started from different vertices of (the same) $\omega$. Because of this the above methods do not apply. Instead we need to use that $\omega$ is \emph{statistically transitive}. 

To this end, we also introduce a killed version of the $\lambda$--biased random walk which is modified at the root, whose quenched law $\Pm^*_{\omega}$ is given by
\begin{align*}
	\Pm^*_{\omega}(S_{n+1} = i | S_{n} =\rt) 	&=\frac{1}{\lambda + \kappa_{\rt}},\quad i=1,\ldots, \kappa_{\rt},\\
	\Pm^*_{\omega}(S_{n+1} = \dagger | S_{n} =\rt) &=\frac{\lambda}{\lambda +\kappa_{\rt}},\\ 
	\Pm^*_{\omega}(S_{n+1} = x | S_{n} =y) 	&= \Pm_{\omega}(S_{n+1} = x | S_{n} =y) \hbox{ for } y\neq \rt\, ,
\end{align*}
where $\dagger$ is an additional cemetery point which is an absorbing state. 
Define the corresponding BRWs with quenched and annealed laws denoted by $\mathbf{P}^*_{\omega}$ and $\PP^*$ respectively (particles absorbed at $\dagger$ do not branch). 

\begin{proof}[Proof of Proposition \ref{prop:0-1-law-for-liminf} in the transient phase.] 
	We focus on the phase of $(m,\lambda)$ in Theorem \ref{thm:ls} where a $\lambda$--biased BRW is $\Pw$-almost surely transient for 
	$\mathtt{BGW}$--almost every $\omega$.  
	
	Consider a BRW $X^*$ of law $\mathbf{P}^*_{\omega}$. 
	We can couple this killed BRW $X^*$ with a BRW $X$ of law $\mathbf{P}_{\omega}$ by adding extra randomness. More precisely, for each particle $u$ of $X$ situated at the root with $X(u) = \rt$, we kill each of its children with probability $q = \frac{\lambda}{\lambda+ \kappa_{\rt}}$, independently of everything else. The resulting BRW has the same law as $X^*$ restricted to $\omega$ without $\dagger$. 
	 Let $N$ be the total number of particles of $X$ that are children of particles at the root $\rt$, that is $N = \#\{uj\in \mathbb{T}\colon X(u) =\rt, j\in \NN\}$. Then $N$ is $\mathbf{P}_{\omega}$-a.s.\ finite, due to the transience assumption. 	 
	 Set  
	 \begin{equation*}
	 		 B_a:= A_a^c= \Big\{ \liminf_{n \to \infty} \max_{\vert u \vert =n} \frac{\vert X(u) \vert}{n} \geq
	 	a\Big\}.
	 \end{equation*}
	 If $	\mathbf{P}_{\omega} \left(B_a \right)>0$, then we have  
		\begin{equation*}
		    		\mathbf{P}^*_{\omega}(B_a )
		    	\geq	\mathbf{P}^*_{\omega} \big(	B_a;\; \text{no particles are killed} \big)
				=
				\mathbf{E}_{\omega} \left[\ind{	B_a} (1-q)^{N}
				\right]  
				>0. 
			\end{equation*}
		
			 Now we view $\Upsilon(\omega) = \mathbf{P}_\omega \left( B_a \right)$ and $\Upsilon^*(\omega) = \mathbf{P}_{\omega}^* \left(B_a \right)$ as functions of the random variable $\omega$ under $\mathtt{BGW}$. 
	The previous considerations showed that $\Upsilon^*(\omega)>0$, $\mathtt{BGW}$--almost surely under the assumption that $\Upsilon(\omega)>0$, $\mathtt{BGW}$--almost surely.

			Consider now 
	\begin{equation}
    		\mathcal{W}_k = \left\{ u \in \TT: \vert X(u) \vert=k, \forall v \prec u: \vert X(v) \vert < k \right\},
	\end{equation}
	the set of particles that first hit level $k$ in their genealogical line of descent. Note that for $u,v \in
	 \mathcal{W}_n$ we might have $X(u) = X(v)$ which we would like to avoid. Therefore for each $x\in\{y\in
	 \omega: \exists u\in \mathcal{W}_k \ \text{with} \ X(u)=y\}$ pick a representative $u\in \mathcal{W}_k$ with
	 $X(u)=x$. Let $\Tilde{\mathcal{W}}_k \subseteq \mathcal{W}_k$ be the set of representatives. 
	This means that for $u,v\in \widetilde{\mathcal{W}}_n$ we have $X(u) \neq X(v)$.

	Fix $\omega$. We observe that for each $k$
	\begin{align*}
    		1-\Upsilon(\omega) = \mathbf{P}_\omega \left( \liminf_{n \to \infty} \max_{\vert u \vert =n} \frac{\vert
		 X(u) \vert}{n} <a\right) 
    		\leq \mathbf{P}_\omega \left( \forall u \in \widetilde{\mathcal{W}}_k: \liminf_{n \to \infty}
		 \max_{\substack{\vert v \vert =n \\ u \prec v}} \frac{\vert X(v) \vert}{n} < a\right).
	\end{align*}
	Next, observe that, conditional on positions $(X(u), u\in \widetilde{\mathcal{W}}_k)$, the descendant
 	populations of distinct particles $u\in \widetilde{\mathcal{W}}_k$ evolve independently. This means that by applying the Markov property we obtain
	\begin{equation*}
    		\mathbf{P}_\omega \left( \forall u \in \widetilde{\mathcal{W}}_k: \liminf_{n \to \infty} 
		\max_{\substack{\vert v \vert =n \\ u \prec v}} \frac{\vert X(v) \vert}{n} < a\right) = \mathbf{E}_{\omega}
		 \left[ \prod_{u \in \widetilde{\mathcal{W}}_k} \mathbf{P}_{\omega,{X(u)}} \left( \liminf_{n \to \infty}
		 \max_{\vert v \vert =n } \frac{\vert X(v) \vert}{n} < a\right)\right],
	\end{equation*}
	where $\mathbf{P}_{\omega,x}$ corresponds to the process started from a single particle located at $x$.
	For $x\in \omega$, we let $F_\omega(x)$ be the fringe subtree rooted at $x$, that is
	$F_\omega(x) = \left\{y \in \omega: y\succeq x \right\}.$ We introduce killing again, this time we modify 
	$ \mathbf{P}_{\omega,{X(u)}} $ by killing all particles that cross the edge $(X(u),X(u)_-)$, the edge from
	$X(u)$ to its parent. Denote the corresponding probability measure by $ \mathbf{P}_{\omega,{X(u)}}^*$.
	As before, the process with killing can be realised as a subset of particles of the process without killing, and thus the maximal displacement, $\max_{\vert v \vert=n} \vert X(v) \vert$ in the
	process without killing stochastically dominates the maximal displacement in the process with killing.
	Using the convention $\max \emptyset = -\infty$ we obtain
	\begin{align*}
    		\mathbf{P}_{\omega,{X(u)}} \left( \liminf_{n \to \infty} \max_{\vert v \vert =n } \frac{\vert X(v) \vert}{n} 
		< a\right) &\leq \mathbf{P}_{\omega,X(u)}^* \left( \liminf_{n \to \infty} \max_{\vert v \vert =n } 
		\frac{\vert X(v) \vert}{n} < a\right) \\
		&=  \mathbf{P}_{F_\omega(X(u))}^* \left( \liminf_{n \to \infty} \max_{\vert v \vert =n } 
		\frac{\vert X(v) \vert}{n} < a\right) = 1-\Upsilon^*(F_\omega(X(u))).
	\end{align*}
	This yields the inequality
	\begin{equation*}
    		1- \Upsilon(\omega) \leq \mathbf{E}_{\omega} \left[ \prod_{u \in \widetilde{\mathcal{W}}_k} 
		\bigg(1-\Upsilon^*\big(F_\omega(X(u))\big) \bigg)\right].
	\end{equation*}
	Going to the annealed law, i.e.\ integrating over $\omega$, and rearranging this means
	\begin{equation*}
    		\EE_{\mathtt{BGW}}\left[ \Upsilon(\omega) \right] \geq 1 - \EE 
		\left[ \prod_{u \in \widetilde{\mathcal{W}}_k} \bigg(1-\Upsilon^*\big(F_\omega(X(u))\big) \bigg) \right].
	\end{equation*}
	
	Note that conditional on $ \widetilde{\mathcal{W}}_k$ and $(X(u), u\in \widetilde{\mathcal{W}}_k)$ ,  the $(F_\omega(x), \vert x \vert =k )$ are \emph{i.i.d.}\ 
	$\mathtt{BGW}$--trees under $\EE$. This means that we can improve the inequality to 
	\begin{equation*}
    		\EE_{\mathtt{BGW}}\left[ \Upsilon(\omega) \right] \geq 1 - \EE
		\left[ \prod_{u \in \widetilde{\mathcal{W}}_k} \bigg(1-\Upsilon^*\big(F_\omega(X(u))\big) \bigg) \right] = 1 - \EE
		\left[ \EE_{\mathtt{BGW}} \left[ 1-\Upsilon^*(\omega) \right]^{\vert \widetilde{\mathcal{W}}_k \vert }  \right] .
	\end{equation*}
	Lastly, we claim that as $k\to \infty$ we have $\vert \widetilde{\mathcal{W}}_k \vert \to \infty $, we show this in Lemma \ref{lemma:spread} below. 
From this and $\Upsilon^*(\omega)>0$, $\mathtt{BGW}$--almost surely
we obtain that
	$\EE_{\mathtt{BGW}}\left[ \Upsilon(\omega)\right] \geq 1.$
	This implies that $ \Upsilon(\omega) =1$ for $\mathtt{BGW}$--almost every $\omega$ because 
	$\Upsilon(\omega)\leq 1$ as it is a probability.  
	\end{proof}
	
	\begin{lemma}\label{lemma:spread}
	In the setting of the proof of Proposition \ref{prop:0-1-law-for-liminf} in the transient regime, we have that $\vert \widetilde{\mathcal{W}}_k \vert \to \infty $, $\PP$--almost surely as $k\to \infty$.
	\end{lemma}
	
	\begin{proof}
	
	We consider a tagged particle like in the proof of Theorem~\ref{thm:ls}: We look at the random walk $(S^*_n, n\ge 0)$ given by the ancestral lineage $(S^*_0 = X (\varnothing), S^*_1 = X(1), S^*_2 = X(11), \ldots)$. Let $1_n$ denote the tagged particle in the $n$-th generation so that $S_n^*=X(1_n)$. This is a $\lambda$--biased random walk. Inspect the trace of $S^*$, $(S^*_n, n\geq 0)$. Consider the set 
	\begin{equation*}
		R = \big \{x \in \omega: \exists n \in \mathbb{N}: S^*_n=x, \forall \ell \in \mathbb{N}: S^*_\ell \notin F_{\omega}(x)\backslash \{x\} \big\}.
	\end{equation*}
	That is, $x\in R$ if $S^*$ encounters $x$ but at the same time does not explore the subtree associated with $x$. An event which could cause $x \in R$ is the following: $S^*$ encounters $x_-$, the parent of $x$, moves to $x$, moves back to $x_-$, moves to a different child of $x_-$, and then never visits $x_-$ again. By transience of $S^*$ we can see that $\vert R \vert = \infty$, $\PP$--almost surely. 
	Indeed, for each vertex $x \in \{S^*_i, i \ge 0\}$, consider the last time $n$ such that $S^*_n = x$, which is almost surely finite by transience. After time $n$, one of two mutually exclusive situations occurs:
	\begin{enumerate}
		\item  The walk $\{S^*_i, i \ge n+1\}$ never returns to the fringe tree $F_\omega(x)$ rooted at $x$. 
		\item The walk never leaves $F_\omega(x)$. 
	\end{enumerate}
	In the first case, there must be a vertex $y$ on $F_\omega(x)$ that belongs to $R$.
	Now suppose, for a contradiction, that $|R| < \infty$. Then after some finite time $N$, the walk must always fall into the second situation. Consequently, from time $N$ onward the walk $\{S^*_i, i \ge N\}$ would be confined to a single ray of the tree $\omega$, an event that occurs with probability zero. So we have $\PP(|R| = \infty)=1$.

	For $x\in \omega$, we define an event $\mathtt{split}(x)$ for the branching random walk: 
	\begin{enumerate}
		\item We have $x\in R$, let $\tau_x$ be the first time $S^*$ hits $x$. 
		\item The particle $1_{\tau_x}$ (which corresponds to $S^*_{\tau_x}$) has at least $2$ children.
		\item The particle $v=1_{\tau_x}2$ (this is the second offspring particle of $1_{\tau_x}$ which does not correspond to $S^*_{\tau_x+1}$) moves to a child of $x$. 
		\item The lineage corresponding to $v$ (that is particles of the form $v1_{\ell}$ for $\ell \geq 0$) never revisits $x$. 
	\end{enumerate}
	See Figure \ref{fig:split} for an illustration of this event. 
	We use these events for a lower bound on $\liminf_{k \to \infty}  \vert \widetilde{\mathcal{W}}_k \vert$,
	\begin{equation}\label{eq:lower_bound_W}
		\liminf_{k \to \infty} \vert \widetilde{\mathcal{W}}_k \vert \geq  \sum_{x\in R }\ind{\mathtt{split}(x)},
	\end{equation}
	that is, whenever $\mathtt{split}(x)$ happens, $\liminf_{k \to \infty}  \vert \widetilde{\mathcal{W}}_k \vert$ increases by $+1$.
	To estimate the probabilities of $\mathtt{split}(x)$, we define a sequence of $\sigma$--algebras $(\mathcal{G}_n)_{n \geq 1}$. To this end, conditional on $\left(S_\ell^*\right)_{\ell\geq 0}$ we enumerate $R$ so that we can write $R = (x_1, x_2, \ldots)$, $\mathbb{P}$--almost surely. We set
	\begin{equation*}
	\mathcal{G}_n = \sigma\left\{\left(S_\ell^*\right)_{\ell\geq 0} ; (\kappa_{S_\ell^*})_{\ell \geq 0}; \mathtt{split}(x_i) \text{ for }1\leq i<n\right\}.
	\end{equation*}
	
	\begin{figure}[tbh]
		\begin{center}
			\includegraphics[width=0.4\textwidth]{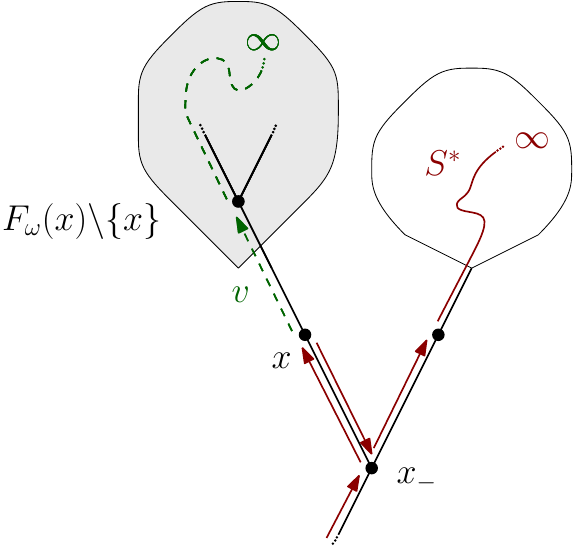}
			\caption{An illustration of the event $\mathtt{split}(x)$: the tagged particle $S^*$ (in solid red) visits $x$ but not $F_\omega(x)\backslash \{x\}$ (in light grey). While $S^*$ visits $x$, the BRW produces a new particle $v$ (in dashed green) that moves to a child of $x$ and then never revisits $x$.}
			\label{fig:split}
		\end{center}
	\end{figure}
	
	  Because by our choice of $x_n$ the random walk $S^*$ never explores  $F_{\omega}(x_n)\backslash \{x_n\}$, we observe that, conditionally given $\mathcal{G}_n$, the family $(F_{\omega}(x_n i ), 1\le i\le \kappa_{x_n})$ are i.i.d.\  $\mathtt{BGW}$-trees. Further, for $i<n$, $\mathtt{split}(x_i)$ does not depend on $F_{\omega}(x_n)\backslash \{x_n\}$. Therefore we have
	\begin{align*}
		&\mathbb{P}\left(\mathtt{split}(x_n) \middle \vert \mathcal{G}_n\right)\\
		&\qquad = (1-\mu_1)\frac{\kappa_{x_n}}{\lambda+\kappa_{x_n}}\mathbb{P}\left(\text{$\lambda$--biased RW $(S_k)$ never visits $x_n$ starting from a uniformly chosen child of $x_n$}\middle\vert \mathcal{G}_n \right)\\
		&\qquad \geq (1-\mu_1)\frac{\kappa_{x_n}}{\lambda+\kappa_{x_n}}\mathbb{P}^*\left(\tau_{\dagger} = \infty \right),
	\end{align*}
	where in the second step we used that, under the conditional probability, the particle $v$ sees a $\mathtt{BGW}$--distributed tree; switching from $\mathbb{P}$ to $\mathbb{P}^*$  (modified at the root) then gives the probability as $\mathbb{P}^*(\tau_\dagger = \infty)$ where $\tau_\dagger = \inf\{k\ge 0\colon S_k =\dagger\}$ is the hitting time to $\dagger$. 
	
	Lastly, we estimate $\frac{\kappa_{x_n}}{\lambda + \kappa_{x_n}} \geq \frac{1}{\lambda + 1}$ and therefore for all $n\geq 1$ we have
	\begin{equation*}
		\mathbb{P}\left(\mathtt{split}(x_n) \middle \vert \mathcal{G}_n\right) \geq (1-\mu_1)\frac{1}{\lambda +1} \mathbb{P}^*\left(\tau_\dagger = \infty \right)>0,
	\end{equation*}
	because we are in the regime where the $\lambda$--biased random walk is transient. 
	This estimate is uniform in $n$, and does not depend on $\mathcal{G}_n$. Hence
	\begin{equation*}
		\sum_{n=1}^\infty \mathbb{P}\left(\mathtt{split}(x_n) \middle \vert \mathcal{G}_n\right) \geq \sum_{n=1}^\infty (1-\mu_1)\frac{1}{\lambda +1} \mathbb{P}^*\left(\tau_\dagger = \infty \right) = \infty.
	\end{equation*} 
	A conditional version of the Borel--Cantelli Lemma (see for example \cite[Cor. 9.21]{kallenberg_foundations_2021}) tells us that therefore also $\sum_{n=1}^\infty \ind{\mathtt{split}(x_n)} = \infty$, $\mathbb{P}$--almost surely. Together with \eqref{eq:lower_bound_W}, we obtain that $\vert \widetilde{\mathcal{W}}_k \vert \to \infty$ as $k \to \infty$, $\mathbb{P}$--almost surely. 
	\end{proof}

\subsection{Annealed probability of fast particles, proof of Proposition \ref{prop:annealed_survival_proba}}\label{sec:annealed_survival_proba}

The goal of this section is to show Proposition \ref{prop:annealed_survival_proba}. That is we want to show that, for any $a<v_{\lambda,m}$,
\begin{align*}
		&\PP \left( \liminf_{n\to\infty} \max_{\vert u \vert = n} \frac{\vert X(u) \vert}{n} \ge  a\right) > 0.
\end{align*} 


The crucial idea is to look for trajectories that move from level $D_{(i-1)n}$ to level $D_{in}$ without revisiting level $D_{(i-1)n}$ and that move quickly, that is in time less than $n/a$. This allows us to use an approximation with an $\NN_0$--valued branching process in random environment. If we can show that this branching process survives with positive probability we are done.

We start by introducing some notation. For any $u\in \TT$ and $n\in \NN$ let 
\begin{equation*}
	T^u_n:=\inf \{k\le |u|: ~ |X(u_k)|  = n\} \quad \text{and} \quad \mathcal{L}_n:= \{ u_{T^u_n}, u\in\TT \},
\end{equation*} 
that are the first hitting time of $D_n=\{x\in \omega: \vert x \vert=n\}$ along the ancestral lineage of $u$ and the collection of particles that hit $D_n$ for the first time in their lines of descent respectively. 

Specifically, we fix 
$n>0$. Set $\mathcal{Z}^{(n)}_{0}= \{\varnothing\}$ and 
\begin{equation}\label{eq:Z1}
\mathcal{Z}^{(n)}_{1} = \{u\in  \mathcal{L}_n \colon |u| \le a^{-1} n,X(u_k)\ne \rt, \forall k= 1,\ldots, |u|   \} . 
\end{equation}
For $i\in \NN$, we define recursively a family of particles 
\begin{equation*}
	\mathcal{Z}^{(n)}_{i+1} = \left\{u\in \mathcal{L}_{(i+1)n}: u_{T^u_{i n}}\in \mathcal{Z}^{(n)}_i, |u|- T^u_{i n}\le a^{-1} n,  \vert X(u_k) \vert > in, \, T^u_{in} < k \leq \vert u \vert \right\}.
\end{equation*}
These are the particles that, for $j \leq i$, move from level $D_{jn}$ to $D_{(j+1)n}$ quickly without revisiting $D_{jn}$. 
We claim that particles like these exist for all $i$ with positive probability. 

\begin{lemma}\label{lem:anneal0}
	There exists $n_0>0$, such that for every $n\geq n_0$, 
	\begin{equation*}
	\PP(\#\mathcal{Z}^{(n)}_{i}>0 ~\text{for all}~ i\ge 1)> 0.  
	\end{equation*}
\end{lemma}

\begin{proof}
	We want to compare $\#\mathcal{Z}^{(n)}_{i}$ to an $\NN_0$--valued branching process in random environment. We stress
	that $(\#\mathcal{Z}^{(n)}_{i},i\geq 0)$ itself is not a branching process under the annealed law $\PP$. 
	The reason for this is that two particles that are counted in $\mathcal{Z}^{(n)}_{i}$ may or may not have been located at the same
	vertex in $D_{(i-1)n}$. In this case they depend differently on different regions of $\omega$ and do not have a consistent structure of independence/dependence. 
	See Figure \ref{fig:branching_process_approximation} for an illustration. 
	
	We compare this to a process where all particles start from the same vertex in $\omega$ at times $(in, i\geq 0)$. Informally, this increases the dependence between particles as they now all use the same environment. The increased dependence should lead to a higher probability of extinction. This also means that if this modified process survives with positive probability, so should the branching random walk.

	\begin{figure}[t]
		\begin{center}
		\includegraphics[width=0.5\textwidth] {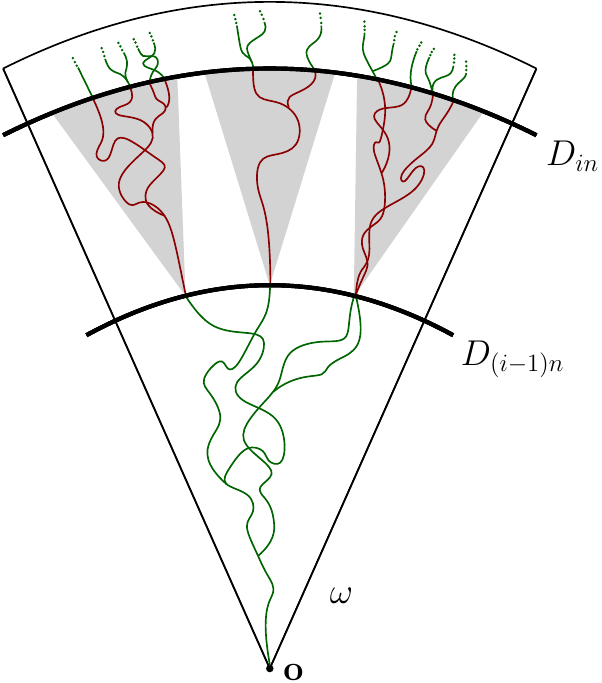}
		\caption{An illustration of the trajectories counted in $\mathcal{Z}_i^{(n)}$. Observe that the trajectories are allowed to backtrack but not below level $D_{(i-1)n}$ so that the trajectories in different cones are independent. Under $\PP$ the grey cones are independent copies of $\omega$. Each cone may have multiple particles starting in it, we compare this to a process where all particles use the same cone.}
		\label{fig:branching_process_approximation}
		\end{center}
	\end{figure}
	
	We start by choosing $n_0$ large enough so that for all $n \geq n_0$ we have 
	\begin{equation}\label{eq:supercritical}
		\EE_{\mathtt{BGW}} \left[\log \Ew\big[\#\mathcal{Z}^{(n)}_{1}\big] \right] >0.
	\end{equation}
We prove that such $n_0$ exists in Lemma \ref{lem:log}  below. Fix $n\geq n_0$ and drop the superscript to write 
	$Z_i := \#\mathcal{Z}^{(n)}_{i}$. We define an auxiliary branching process in random environment $(Z_i', i\geq 0)$. 
	Let $\vec{\omega}:=(\omega_{i})_{i\ge 0}$ be a family of \emph{i.i.d.}\ BGW--trees with distribution 
	$\mathtt{BGW}$. Set $Z'_0=1$ and define inductively
	\begin{equation*}
		Z'_{i+1} := \sum_{j=1}^{Z'_i} \xi'_{i,j},
	\end{equation*} 
	where each $\xi'_{i,j}$ is an independent copy of $Z_1$ under the quenched law $\Ps_{\omega_{i}}$ using
	$\omega_{i}$ as environment. Denote the quenched law for $(Z_i', i\geq 0)$ by 
	$\mathbf{P}_{\vec{\omega}}'$ and the annealed law by $\PP'$.
	The consequence of \eqref{eq:supercritical} is that this is a supercritical branching process in random
	 environment, see e.g.\ \cite[Section VI.5]{Ath-Ney}. In particular, the annealed survival probability is strictly 
	positive, i.e.\ there exists $c>0$, such that
	\begin{equation*}
		\PP' (Z'_i>0 ~\text{for all}~ i\ge 1)\geq c.
	\end{equation*}
	We set for $i\ge 0$ and $\ell \ge 1$
	\begin{equation*}
		\theta_{i,\ell} :=\PP \left(Z_\ell=0 \middle\vert Z_0 = i \right) \quad \text{and} \quad \theta_{i,\ell}' :=\PP' \left(Z_\ell'=0 \middle\vert Z_0'= i \right),
	\end{equation*}
	the annealed probabilities of extinction after $\ell$ steps when starting with $i$ particles.
	Note that for both $\PP$ and $\PP'$, all starting particles use the same environment. This means we have
	\begin{equation*}
		\theta_{i,\ell} =\EE_{\mathtt{BGW}}  \left[ \Pw(Z_\ell= 0)^i \right] \quad \text{and} \quad \theta_{i,\ell}' =\EE_{\mathtt{BGW}}  \left[ \mathbf{P}_{\vec{\omega}}'(Z'_\ell= 0)^i \right].
	\end{equation*}
	By construction, $Z_1$ and $Z_1'$ have the 
	same distribution and therefore for all $i\geq 0$ we have $\theta_{i,1} = \theta_{i,1}'$. In particular,
	\begin{equation*}
		\theta_{i,1}' \geq \theta_{i,1} \qquad \text{for all} ~ i\geq 0.
	\end{equation*}
	We now show by induction that for all $\ell\geq 1$ we have
	\begin{equation*}
		\theta_{i,\ell}' \geq \theta_{i,\ell} \qquad \text{for all} ~ i\geq0.
	\end{equation*}
	For any $i,j\ge 1$, Jensen's inequality leads to\footnote{The consequence of Jensen here is $\EE[X^i]\leq \EE[X^{i+j}]^{\frac{i}{i+j}}$ and $\EE[X^j]\leq \EE[X^{i+j}]^{\frac{j}{i+j}}$, hence $\EE[X^i]\EE[X^j]\leq \EE[X^{i+j}]$.}
	\begin{equation*}
		\theta_{i+j,\ell} =\EE_{\mathtt{BGW}}  \left[ \Pw(Z_\ell= 0)^{i+j} \right] \ge 
		\EE_{\mathtt{BGW}} \left[  \Pw(Z_\ell= 0)^{i} \right]
		\EE_{\mathtt{BGW}} \left[ \Pw(Z_\ell= 0)^{j} \right]
		=\theta_{i,\ell} \theta_{j,\ell}.
	\end{equation*}
	Inductively, we also have for $a_1,\ldots,a_k \in \NN^k$ with $\sum_{i=1}^k a_i = L$ that
	\begin{equation}\label{eq:conseq_jensen}
		\prod_{i=1}^k \theta_{a_i,j} \leq \theta_{L,j}.
	\end{equation}
	We want to apply the branching property. For $u\in \mathcal Z_1$, we define $(Z_i(u), i \ge 0)$ correspondingly for the subpopulation $(X(uv), uv\in \mathbb{T})$. For $x\in D_n$ let 
	$$Z_{1,x} =\#\{u \in \mathcal Z_1: X(u)=x\},$$ 
	the number of particles in $\mathcal Z_1$ that are located at $x$. 
	Note that $\sum_{x\in D_n}Z_{1,x} = Z_1$. With this (and using the convention that the empty product is $1$) we have for any $\ell\geq 1$ and $i\ge 0$,
	\begin{align*}
		\Pw \left( Z_{\ell+1} = 0\mid Z_0 = i\right) 
		& = \Pw \left( Z_{1} = 0~\text{or}	
		~\forall u\in \mathcal{Z}_1:~\text{the corresponding}~Z_\ell(u) = 0 \,\middle|\, Z_0 =i \right) \\
		&=\mathbf{E}_{\omega} \bigg[  \prod_{x\in D_n : Z_{1,x}>0} \mathbf{P}_{F_\omega(x)} \left( Z_{\ell} = 0\right)^{Z_{1,x}} \,\bigg|\,  Z_0 = i \bigg].
	\end{align*}
	Then by taking the annealed expectation and by applying  \eqref{eq:conseq_jensen} to $(Z_{1,x},x\in D_n)$
	\begin{align*}
		\theta_{i,\ell+1} = \EE \bigg[ \prod_{x\in D_n : Z_{1,x}>0 } \theta_{Z_{1,x}, \ell}  \,\bigg|\,  Z_0 = i  \bigg] 
		\le\EE\left[  \theta_{Z_1, \ell}\,\bigg|\,  Z_0 = i   \right].
	\end{align*}
Recall that $Z_1$ and $Z_1'$ have the same annealed distribution and use the induction hypothesis $\theta_{i,\ell}' \geq \theta_{i,\ell}$ for all $i\geq0$, then we have 
\begin{equation*}
	\theta_{i,\ell+1}  \le\EE\left[  \theta_{Z_1, \ell}\,\middle|\,  Z_0 = i   \right] =\EE'\left[  \theta_{Z'_1, \ell}\,\middle|\,  Z'_0 = i   \right] \le   \EE'\left[  \theta'_{Z'_1, \ell}\,\middle|\,  Z'_0 = i   \right].
\end{equation*}
Lastly, we observe that $\theta'_{i,\ell+1}=\EE'\left[  \theta'_{Z'_1, \ell}\,\middle|\,  Z'_0 = i   \right]$ by the Markov property for $(Z_i', i\geq 1)$. 
Hence we have shown by induction 
\begin{equation*}
	\theta_{i,\ell}' \geq \theta_{i,\ell} \qquad \text{for all} ~ i\ge 0,\ell\geq 1.
\end{equation*}
	We conclude the proof using monotone convergence, 
	\begin{equation*}
		\PP\left(Z_i>0 ~\text{for all}~ i\ge 1\right) = 1- \lim_{\ell\to \infty} \theta_{1,\ell} 
		\ge 1- \lim_{\ell\to \infty} \theta'_{1,\ell } = \PP'\left(Z'_i>0 ~\text{for all}~ i\ge 1\right)\geq c >0.
	\end{equation*}
\end{proof}

It remains to prove \eqref{eq:supercritical}, see Lemma \ref{lem:log} below, but before doing so, let us show how one proves Proposition \ref{prop:annealed_survival_proba} assuming Lemma \ref{lem:anneal0}.

\begin{proof}[Proof of Proposition \ref{prop:annealed_survival_proba}] 
	We first translate Lemma~\ref{lem:anneal0} about hitting times into a statement in terms of the maximal displacement. Let us fix $n\ge n_0$ as given in Lemma~\ref{lem:anneal0}. 
	Take $\tilde{a}\in (a,v_{\lambda,m})$ close enough to $v_{\lambda,m}$, such that $\frac{v_{\lambda,m}-\tilde{a}}{\tilde{a}-a}<v_{\lambda,m}$. 
	 Then we have $\tilde{a}-a > \frac{v_{\lambda,m}-\tilde{a}}{v_{\lambda,m}}$. Let $\varepsilon>0$ be small enough such that  $\tilde{a} -a > \frac{v_{\lambda, m}-\tilde{a}+\varepsilon}{v_{\lambda, m}+ \varepsilon}$. 
	
	Recall the upper bound \eqref{eq:conclusion_upper_bound}: for $\mathtt{BGW}$-a.e. $\omega$, 
		\[
	\Pw \left(\limsup_{j \to \infty} \max_{|v| = j } \frac{|X(v)|}{j} \le v_{\lambda, m}  \right)= 1. 
	\]
	Then there exists a random $n_1(\varepsilon,\omega)$, such that for all $j>n_1(\varepsilon,\omega)$, 
	\begin{equation}\label{eq:upperbdd}
		\max_{|v| = j } |X(v)| \le j (v_{\lambda, m}+ \varepsilon).
	\end{equation}
	
	Consider $(\mathcal{Z}^{(n)}_i, i\ge 0)$ as in Lemma~\ref{lem:anneal0} associated with $\tilde{a}$. 
	For $u\in \mathcal{Z}^{(n)}_{\lfloor \tilde{a} i\rfloor}$, we  have $|X(u)| = \lfloor \tilde{a} i\rfloor n $
	and by the upper bound \eqref{eq:upperbdd}, for all $i$ such that $\lfloor\tilde{a} i\rfloor n >n_1(\varepsilon,\omega)$,  
	\begin{equation}\label{eq:upperbdd-1}
	|X(u)| = \lfloor \tilde{a} i\rfloor n \le \max_{|v| = |u| } |X(v)| \le |u| (v_{\lambda, m}+ \varepsilon).
\end{equation}
Then $ |u|  \ge \frac{\lfloor \tilde{a} i\rfloor n}{v_{\lambda, m}+ \varepsilon}$. 
	By considering the descendants of $u$ at generation $in$, we	observe that 
	\[
	\max_{|v| = i n } |X(v)|
	>\lfloor \tilde{a} i\rfloor n - \Big(in -\frac{\lfloor \tilde{a} i\rfloor n}{v_{\lambda, m}+ \varepsilon}\Big)  
	\ge \Big( \tilde{a} - \frac{v_{\lambda, m}+ \varepsilon-\tilde{a}}{v_{\lambda, m}+ \varepsilon} \Big) i n 
		-\Big(1+\frac{1}{v_{\lambda, m}+ \varepsilon}\Big) n.
	\]
	By our choice of $\tilde{a}$ and $\varepsilon$, we have $\tilde{a} - \frac{v_{\lambda, m}+ \varepsilon-\tilde{a}}{v_{\lambda, m}+ \varepsilon}>a$. Thus, for all $i\in \NN$ large enough, provided that $\mathcal{Z}^{(n)}_{\lfloor \tilde{a} i\rfloor}\ne \emptyset$, we have
	$\max_{|v| = i n } |X(v)|
	> a i n$. 
	It follows from	Lemma~\ref{lem:anneal0} that provided $n$ is large enough, with positive $\PP$-probability there exists  $u \in \mathcal{Z}^{(n)}_{\lfloor \tilde{a} i\rfloor}$ for all $i$ and thus  that 
		\begin{equation*}
				\PP\left(\liminf_{i \to \infty} \max_{|u| = i n } \frac{|X(u)|}{i n} \ge a  \right) \ge 	\PP(\#\mathcal{Z}^{(n)}_{i}>0 ~\text{for all}~ i\ge 1)> 0.  
			\end{equation*}
		This proves the statement for a subsequence of times.
		To complete the proof, we notice that, for $j\in [in, (i+1)n)$ with $i\ge 0$,  
		\[
		\max_{|v| = j } \frac{|X(v)|}{j} \ge \max_{|u| = (i+1) n } \frac{|X(u)| - n}{(i+1) n}. 
		\]
		Letting $i\to \infty$, we conclude that
		\[
				\PP\left(\liminf_{j \to \infty} \max_{|v| = j } \frac{|X(v)|}{j} \ge a  \right)\ge 			\PP\left(\liminf_{i \to \infty} \max_{|u| = i n } \frac{|X(u)|}{i n} \ge a  \right) > 0. 
		\]
\end{proof}

It now remains to prove Lemma \ref{lem:log}, see below, that we have used in the proof of Lemma \ref{lem:anneal0} for Equation \eqref{eq:supercritical}. 
\begin{lemma}\label{lem:log}
  	Let  $v_{\lambda} < a< v_{\lambda, m}$. 
  		Recall from \eqref{eq:Z1} that $\mathcal{Z}^{(n)}_1:= \{u\in \mathcal{L}_n \colon |u| \le a^{-1} n,X(u_k)\ne \rt, \forall k= 1,\ldots, |u| \} $. 
  	It holds that
	\begin{equation}
	\liminf_{n\to \infty}\EE_{\mathtt{BGW}} \left[\frac{a}{n}\log \Ew\big[\# \mathcal{Z}^{(n)}_1 \big] \right] >0.
	\end{equation}
	In particular, there is $n^* \in \NN$ such that for all $n\geq n^*$ we have
	\begin{equation}\label{eq:lem:log}
	\EE_{\mathtt{BGW}} \left[\log \Ew\big[	\# \mathcal{Z}^{(n)}_1   \big] \right] >0. 
	\end{equation}
\end{lemma}

\begin{proof}
	For $v\in \TT$ with $\vert v\vert\ge n$, we write $\tau_\rt^v > n$ if $\forall 0< k\leq n$  
	we have $X(v_k)\neq \rt$. Similarly, for the $\lambda$--biased random walk $(S_j, j\geq 0)$ we set $T_n = \inf\{j\ge 0\colon |S_j|=n  \}$ and $\tau_{\rt} = \inf\{j > 0\colon S_j = \rt  \}$.  
	
	Fix any $\varepsilon>0$. By the many-to-one formula \eqref{eq:many_to_one},
	\begin{align*}
		  & \Ew\left[	\# \{v\in \TT\colon |v|= \lceil a^{-1} n\rceil , (a+\varepsilon)^{-1}n<T^v_{n } \le a^{-1} n  , \tau^v_\rt >T^v_{n}  \} \right] \\
		& = m^{ \lceil a^{-1}  n\rceil } \Pws\big( (a+\varepsilon)^{-1}n< T_{n } \le a^{-1} n, \tau_\rt >T_{n}  \big). 
	\end{align*} 
On the other hand, by the branching property at the stopping line 
$\mathcal{L}_{n}=\{u_{T^u_{n }}, u\in \TT \}$, we have 
\begin{align*}
	&\Ew\left[\# \{v\in \TT\colon  |v|= \lceil a^{-1} n\rceil ,(a+\varepsilon)^{-1}n<T^v_{n } \le a^{-1} n , \tau^v_\rt >T^v_{n}  \} \right] \\
	&= \Ew\left[	\sum_{u\in \mathcal{L}_n } m^{\lceil a^{-1} n \rceil - |u|} \ind{u\in\mathcal{Z}^{(n)}_1 , |u|>  (a+\varepsilon)^{-1}n} \right] 
		\leq m^{\lceil a^{-1}n \rceil- (a+\varepsilon)^{-1}n}\Ew\big[	\# \mathcal{Z}^{(n)}_1 \big].
\end{align*}
Combining the two observations, taking the logarithm and rearranging leads to
	\begin{align}\label{eq:lem-log-Z1}
	&	\log\Ew\big[	\# \mathcal{Z}^{(n)}_1\big]
		\ge (a+\varepsilon)^{-1} n \log m + \log  \Pws\big( (a+\varepsilon)^{-1}n< T_{n } \le a^{-1} n, \tau_\rt >T_{n}  \big) . 
	\end{align}
 We integrate both sides in $\omega$ with respect to $\mathtt{BGW}$ and let $n\to \infty$.  
It follows from Lemma \ref{lem:log0-bis} below  that
\begin{equation}
	\liminf_{n \to \infty}\EE_{\mathtt{BGW}} \left[ \frac{a}{n} \log \Ew\big[\# \mathcal{Z}^{(n)}_1\big]\right] \ge \frac{a}{a+\varepsilon} \log m -I_{\lambda} (a).\label{eq:rw-ldp}
\end{equation}
Since $I_{\lambda}$ is strictly increasing on $(v_{\lambda},1]$ with $\log m \leq I_{\lambda}(v_{\lambda, m})$, we have  $\log m> I_{\lambda}(a)$, thus the right-hand side is strictly positive for $\varepsilon$ small enough. 
This completes the proof.
\end{proof}

As we have discussed above, it remains to complete the large deviation estimate used in \eqref{eq:rw-ldp}, as stated in the following lemma. 
\begin{lemma}\label{lem:log0-bis}
	Let $(S_j, j\ge 0)$ be the $\lambda$--biased random walk starting from $\rt$. Set $T_n = \inf\{j\ge 0\colon |S_j|=n  \}$ and $\tau_{\rt} = \inf\{j > 0\colon S_j = \rt  \}$. 
	Let $a,b\in (v_{\lambda}, 1]$ with $a<b$. 
	Then we have 
	\begin{equation*}
		\lim_{n\to \infty} \EE_{\mathtt{BGW}}  \Big[ \frac{a}{n}\log  \Pws\big( b^{-1} n< T_{n} \le a^{-1}  n, \tau_\rt >T_{n}  \big)  \Big]  = - I_{\lambda}(a). 
	\end{equation*}
\end{lemma}
\begin{proof}
	The proof relies on several estimates in \cite{DGPZ-ld}.  
	We have by \cite[(4.7) and (4.10)]{DGPZ-ld} that $\mathtt{BGW}$-a.s. $\omega$, 
	\[
	\lim_{n\to \infty} \frac{1}{n}\log\Pws\big(T_{n} \le a^{-1}  n\big) =- \frac{I_{\lambda}(a)}{a}.
	\]
	Since we also know from \cite[(4.5)]{DGPZ-ld} that, for all $n\in \mathbb{N}$, 
	\[
	\Pws\big( T_{n} \le a^{-1}  n, \tau_\rt >T_{n}  \big)
	\ge a n^{-1} \Pws\big( T_{n} \le a^{-1}  n  \big), 
	\]
	it follows that  $\mathtt{BGW}$-a.s. $\omega$, 
	\begin{equation*}
		\lim_{n\to \infty} \frac{1}{n}\log\Pws\big(T_{n} \le a^{-1}  n, \tau_\rt >T_{n}\big) =- \frac{I_{\lambda}(a)}{a}. 
	\end{equation*}
Let $b>a$ with $b\in (v_{\lambda}, 1]$, as $\frac{I_{\lambda}(a)}{a}<\frac{I_{\lambda}(b)}{b}$ by \cite[Page~255]{DGPZ-ld},  we deduce that, for $\mathtt{BGW}$-a.e.\ $\omega$, 
\begin{equation}\label{eq:ld-Tn}
\lim_{n\to\infty}	\frac{a}{n} \log \Pws\big( b^{-1}n< T_{n } \le a^{-1} n, \tau_\rt >T_{n}  \big)
= -I_{\lambda} (a) . 
	\end{equation}
	
	Let us next analyse the probability $\Pws\big( b^{-1} n< T_{n } \le a^{-1} n, \tau_\rt >T_{n}  \big)$ for any $n\in \NN$.  
	Consider the event that the biased random walk $(S_j)$ moves forward in the first step to $1\in D_1$, next oscillates between the two vertices $1$ and $11\in D_{2}$ for $(\lceil b^{-1}  n\rceil -n)/2$ circles, and then moves forward for the next $n-1$ steps along the path from $1$ to $11\cdots 1\in D_n$. 
	For this trajectory of $(S_j)$, we have $\tau_\rt >T_{n}$ and $T_n =\lceil b^{-1}  n\rceil \le a^{-1} n$ for all $n>1/(a^{-1} -b^{-1} )$.  
	This gives a lower bound for every $\omega$,
	\begin{align*}
	\Pws\big( b^{-1}n< T_{n } \le a^{-1} n, \tau_\rt >T_{n}  \big) \geq \frac{1}{\kappa_{\rt}} \left( \frac{1}{\kappa_1 + \lambda} \frac{\lambda}{\kappa_2 + \lambda}\right)^{\frac{\lceil b^{-1}  n\rceil -n}{2}}  \prod_{i=1}^{n-1}\frac{1}{\kappa_i + \lambda},
	\end{align*}
	where $\kappa_{i}$ denotes the outer degree of the vertex $11\cdots 1\in D_i$. 
	Taking logarithms and using the bound $\log(x) \leq x$ for $x > 0$, we have
	\begin{align*}
	&\frac{1}{n} \log \Pws\big( b^{-1}n< T_{n } \le a^{-1} n, \tau_\rt >T_{n}  \big)\\ 
	& \hspace{4cm}  \geq -\frac{\kappa_{\rt}}{n} + \frac{(\lceil b^{-1}  n\rceil/n)-1}{2} \left( -\kappa_1 -\kappa_2- 2\lambda+\log(\lambda)\right)- \frac{1}{n}\sum_{i=1}^{n-1}(\kappa_i + \lambda).
	\end{align*}
Since $(\kappa_{i}, i\ge 1)$ are i.i.d.\ under $\mathtt{BGW}$ with $\mathbb{E}_{\mathtt{BGW}} [\kappa_{i}]= m_{\mathtt{BGW}}<\infty$,  a consequence of the law of large numbers shows that the right-hand side forms a family of uniformly integrable random variables.  
	Then we deduce by the uniform integrability that  the convergence in \eqref{eq:ld-Tn}  also holds in $L^1(\mathtt{BGW})$ and the statement follows. 
\end{proof}


\subsection{Proof of Theorem \ref{thm:max}}\label{sec:main_thms_proofs}

With Propositions \ref{prop:annealed_survival_proba} and \ref{prop:0-1-law-for-liminf} in hand, we turn to the proof of Theorem \ref{thm:max}. That is we show under our assumptions that for $\mathtt{BGW}$--almost every $\omega$ we have
\begin{equation*}
	\lim_{n\to \infty} \frac{1}{n} \max_{|u|=n} |X(u)|  =  v_{\lambda, m}, \qquad \Pw-a.s. 
\end{equation*}

\begin{proof}[Proof of Theorem \ref{thm:max}]
The upper bound part has been proved by \eqref{eq:conclusion_upper_bound}. 
	For the lower bound, most of the work has been done in Propositions 
	\ref{prop:annealed_survival_proba} and \ref{prop:0-1-law-for-liminf}. We consider $\Pwo$ again 
	as introduced in Section \ref{sec:0_1_laws}. For $a>0$ set 
	\begin{equation*}	
		G_a = \left \{\omega\colon \Pwo \left( \liminf_{n \to \infty} \max_{|u|=n} \frac{\vert X(u) \vert }{n} \geq a \right)=0 \right\}.
	\end{equation*}
	Let $1,2,\ldots, \kappa_{\rt}$ be the children of the root in $\omega$, and $\omega^i=F_\omega(i)$ be the subtree rooted at $i = 1,2,\ldots, \kappa_{\rt}$. 
	Then we claim that
	\begin{equation}\label{eq:inherited}
	\{\omega\in G_a\} \subseteq \bigcap_{i=1}^{\kappa_{\rt}} \{\omega^i\in G_a\}.  
	\end{equation}
	To see that, let us consider a BRW under law $\mathbf{P}^*_{\omega}$, starting from one particle at the 
	root. With positive probability, we have one particle in the first generation, say $u$, located at the vertex $i$. 
	If $\omega^i\not\in G_a$, i.e.\ 
	$ \mathbf{P}_{\omega^i}^* (  \liminf_{n \to \infty} \max_{|u|=n} \frac{\vert X(u) \vert }{n} \geq a)>0$, then  
	$ \mathbf{P}_{\omega}^* ( \liminf_{n \to \infty} \max_{|u|=n} \frac{\vert X(u) \vert }{n} \geq a )>0$, 
	i.e.\ $\omega\not\in G_a$. This shows \eqref{eq:inherited} and therefore $G_a$ is an inherited property in the 
	sense of Lemma \ref{lem:inherited}. This  yields that $\mathtt{BGW}(G_a)\in \{0,1\}$. 

	Let $\omega \in G_a^c$, i.e.\ $\mathbf{P}_{\omega}^* ( \liminf_{n \to \infty} \max_{|u|=n} \frac{\vert X(u) \vert }{n} \geq a )>0$. Recall that by a coupling argument we have
	  $$0<\mathbf{P}_{\omega}^* ( \liminf_{n \to \infty} \max_{|u|=n} \frac{\vert X(u) \vert }{n} \geq a )\le \mathbf{P}_{\omega} ( \liminf_{n \to \infty} \max_{|u|=n} \frac{\vert X(u) \vert }{n} \geq a ).$$
	We now conclude the proof by using Propositions \ref{prop:annealed_survival_proba} and \ref{prop:0-1-law-for-liminf}. Let $a<v_{\lambda,m}$. By Proposition~\ref{prop:annealed_survival_proba} we have 
	\begin{equation*}
		0<\PP \left( \liminf_{n\to\infty} \max_{\vert u \vert = n} \frac{\vert X(u) \vert}{n} \ge  a\right) \le \mathtt{BGW}(G_a^c),
	\end{equation*}
	which implies $\mathtt{BGW}(G_{a}^c) =1$. This means that 
	\begin{align*}
		&\Pw\left( \liminf_{n\to\infty} \max_{\vert u \vert = n} \frac{\vert X(u) \vert}{n} \geq a\right) > 0 , \quad \text{for} \quad \mathtt{BGW}-a.e. \ \omega .
	\end{align*}
	This allows us to apply Proposition \ref{prop:0-1-law-for-liminf} which implies
	\begin{align*} \Pw\left( \liminf_{n\to\infty} \max_{\vert u \vert = n} \frac{\vert X(u) \vert}{n} \geq a\right) =1, \quad \text{for} \quad \mathtt{BGW}-a.e. \ \omega.
	\end{align*}
	Together with \eqref{eq:conclusion_upper_bound}, this shows Theorem \ref{thm:max}.
\end{proof}

\subsection{Proof of Theorem \ref{thm:min}}\label{sec:min}
In this section we consider the minimal distance to the root in the transient regime and prove Theorem~\ref{thm:min}. Our proof does not provide full details, as the approach closely follows that of Theorem~\ref{thm:max} for the maximal displacement.

 Switching from maximum to minimum, the trivial direction -- corresponding to \eqref{eq:conclusion_upper_bound} -- now serves as the lower bound: 
\begin{equation}\label{eq:min:easy_direction}
	\Pw\left( \liminf_{n\to\infty} \min_{\vert u \vert =n} \frac{\vert X(u) \vert}{ n}\geq \Tilde{v}_{\lambda,m} \right)=1, \qquad \text{for} ~ \mathtt{BGW}-a.e. \ \omega.
\end{equation}
Indeed, letting $\varepsilon > 0$, by the many--to--one formula \eqref{eq:many_to_one} we have 
\begin{equation*}
	\Pw\left( \min_{|u|=n} \frac{\vert X(u) \vert }{n}\leq \Tilde{v}_{\lambda,m}- \eps \right) \le  \Ew \bigg[\sum_{|u|=n} \ind{|X(u)|\leq  (\Tilde{v}_{\lambda,m} -\eps)n} \bigg] =  m^n \Pws\left( |S_n|\leq \Tilde{v}_{\lambda,m} -\eps \right).    
\end{equation*}
By our choice of $\Tilde{v}_{\lambda,m}$ given in \eqref{velocity-descr_min}, we have $I_{\lambda}(\Tilde{v}_{\lambda,m}-\eps)>\log m$ and it follows from Theorem \ref{lem:ld2} that, for $\mathtt{BGW}$-a.e.\ $\omega$, 
\[
\limsup_{n\to \infty} \frac{1}{n} \log \Pw\bigg( \min_{|u|=n} \frac{\vert X(u) \vert }{n}\leq  \Tilde{v}_{\lambda,m} - \eps \bigg) < 0.
\]
The Borel-Cantelli lemma implies that, for $\mathtt{BGW}$-a.e.\ $\omega$,
\begin{equation*}
	\liminf_{n \to \infty} \min_{|u|=n} \frac{\vert X(u) \vert }{n} \geq  \Tilde{v}_{\lambda,m}-\eps, \qquad \Pw-a.s.
\end{equation*}
As $\eps$ was arbitrary, this yields \eqref{eq:min:easy_direction}.

The upper bound for the minimum, that is, to show the existence of slow particles with small linear velocity close to $\Tilde{v}_{\lambda,m}$, is similar to the lower bound for the maximum. For this we need two statements analogous to Propositions \ref{prop:annealed_survival_proba} and \ref{prop:0-1-law-for-liminf}. For Proposition \ref{prop:0-1-law-for-liminf}, only the transient case is relevant and the arguments requires very little modification when considering the minimal displacement instead of the maximal one. 
The statement  corresponding to Proposition \ref{prop:annealed_survival_proba} is as follows. 
\begin{proposition}\label{prop:annealed_survival_proba-min}
	Under the assumption of Theorem \ref{thm:min} we have for any $\Tilde{v}_{\lambda,m}< a < v_\lambda$, 
	\begin{align*}
		&\PP \left( \limsup_{n\to\infty} \min_{\vert u \vert = n} \frac{\vert X(u) \vert}{n} \le a\right) > 0.
	\end{align*}
\end{proposition}	
We comment on the proof in Section~\ref{sec:annealed_min}. Combining Proposition~\ref{prop:annealed_survival_proba-min} and the zero--one law  completes the proof of Theorem \ref{thm:min}.  

\subsection{Annealed probability of slow particles, proof of Proposition~\ref{prop:annealed_survival_proba-min}}\label{sec:annealed_min}

We now prove Proposition~\ref{prop:annealed_survival_proba-min}. 
Similarly to Lemma~\ref{lem:log0-bis}, the first ingredient we need is a large deviation estimate for the $\lambda$--biased random walk that we show along the lines of \cite{DGPZ-ld}. 

\begin{lemma}\label{lem:log-min}
	Let $(S_j, j\ge 0)$ be a $\lambda$-biased random walk starting from $\rt$. 
	Set $T_n = \inf\{j\ge 0\colon |S_j|=n  \}$ and $\tau_{\rt} = \inf\{j> 0\colon S_j = \rt  \}$. 
	Let $a \in [0, v_\lambda)$. Then under the same assumptions as Theorem~\ref{lem:ld2}, we have 
		\begin{equation*}
		\lim_{n\to \infty}\EE_{\mathtt{BGW}} \Big[ \frac{a}{n}\log \Pws\big(\tau_{\rt}> T_{n} \ge a^{-1}  n\big) \Big]  = - I_{\lambda}(a),
	\end{equation*}
where $I_{\lambda}(a)$ is the rate function in Theorem~\ref{lem:ld2}. 
	In particular, for $a \in (\Tilde{v}_{\lambda,m}, v_\lambda)$ with $\Tilde{v}_{\lambda,m}$ as in \eqref{velocity-descr_min}, such that  $I_{\lambda}(a) < \log m$, 
	\begin{equation}
		\liminf_{n\to \infty}\EE_{\mathtt{BGW}} \Big[ \log m + \frac{a}{n}\log \Pws\big(\tau_{\rt} > T_{n} \ge a^{-1}  n\big) \Big] >0. 
	\end{equation}
\end{lemma}
\begin{proof}
	We deduce by results in \cite{DGPZ-ld} that, for $\mathtt{BGW}$-a.e.\ $\omega$, 
\begin{equation}\label{eq:lemlogmin}
	\lim_{n\to \infty} \frac{1}{n}\log\Pws\big(\tau_{\rt}> T_{n} \ge a^{-1}  n\big) = 
	\lim_{n\to \infty} \frac{1}{n}\log\Pws\big(T_{n} \ge a^{-1}  n\big) =- \frac{I_{\lambda}(a)}{a},
\end{equation}
where the first equality follows from \cite[the display above (5.3)]{DGPZ-ld} and the second one from \cite[(5.3) and (5.5)]{DGPZ-ld}.

	Moreover, for any $n\in \NN$, consider the event that $(S_j)$ moves forward in the first two steps, to the vertex $1\in D_1$ and then $11\in D_{2}$, and then oscillates between the two vertices $1$ and $11$, for the next $\lceil a^{-1}  n\rceil-2$ steps. This gives a lower bound for every $\omega$: 
\begin{align*}
&\frac{1}{n}\log \Pws\big(\tau_{\rt} >T_{n}
\ge a^{-1}  n\big) 
\\
& \ge  \frac{1}{n}\log \frac{1}{\kappa_0}+ \frac{\lceil a^{-1}  n\rceil+1}{2n}\log \left(\frac{1}{\lambda+\kappa_{1}}\right)+\frac{\lceil a^{-1}  n\rceil}{2n} \log \left(\frac{\lambda}{\lambda+\kappa_{2}}\right)
\ge -c_1 - c_2\sum_{i=0}^2 \log \left(\lambda+\kappa_{i}\right),
\end{align*}
where $\kappa_{0}$, $\kappa_{1}$, $\kappa_{11}$ are the outer degrees of the vertices $\rt$, $1$, $11$, respectively, and  $c_1,c_2>0$ are constants that does not depend on $n$ nor $\omega$.
 For $i=0,1,2$, since we have assumed $m_{\mathtt{BGW}} = \mathbb{\mathtt{BGW}}[\kappa_{i}]<\infty$, we have by Jensen's inequality that $\mathbb{E}_{\mathtt{BGW}}[\log \left(\lambda+\kappa_{i}\right)]<\infty$. 
 Therefore, taking expectation $\mathbb{E}_{\mathtt{BGW}}$ to \eqref{eq:lemlogmin}, we are allowed to use the dominated convergence theorem to deduce the desired statement. 
\end{proof}

We turn to the BRW under $\mathbf{P}_{\omega}$. 
As in the maximum case, we consider an embedded discrete-time branching process,  analogous to $\mathcal{Z}^{(n)}$ defined in \eqref{eq:Z1}. 
More precisely, for any particle $u\in \TT$ and $i,n\in \NN$,  write $T^u_{i n}:=\inf \{k \ge 1:  |X(u_k)|  =  i n\}$ for the hitting time of level $D_{in}$ along the ancestral lineage of $u$; let  $\mathcal{L}_{in}:= \{ u_{T^u_{in}}, u\in\TT \}$ be the collection of particles at which the ancestral lineage first hits  $D_{in}$.
Set $\widehat{\mathcal{Z}}^{(n)}_{0} = \{\varnothing\}$.  
Define recursively the set of particles, for $i\ge 0$, 
\begin{equation*}
	\widehat{\mathcal{Z}}^{(n)}_{i+1} = \left\{u\in \mathcal{L}_{(i+1)n}: u_{T^u_{i n}}\in \widehat{\mathcal{Z}}^{(n)}_i, |u|- T^u_{i n}\ge a^{-1} n,  \vert X(u_k) \vert > in, \, T^u_{in} < k \leq \vert u \vert \right\}.
\end{equation*}

Using Lemma~\ref{lem:log-min} and the many-to-one formula, we deduce 
the following analogue of Lemma~\ref{lem:log}  by similar arguments. 
\begin{lemma}\label{lem:max-log}
	Let  $a\in (\widetilde{v}_{\lambda, m}, v_{\lambda})$. Then there is $n'\in \NN$ such that for all $n\geq n'$
	\begin{equation*}
		\EE_{\mathtt{BGW}} \left[\log \Ew\big[\# \widehat{\mathcal{Z}}^{(n)}_1 \big] \right] >0. 
	\end{equation*}
\end{lemma}

Analogously to Lemma~\ref{lem:anneal0} we can now prove that the branching process $(\widehat{\mathcal{Z}}^{(n)}_{i}, i\ge 0)$ survives with positive probability. That is, for every $n\geq n'$ we have
\begin{equation*}
	\PP(\#\widehat{\mathcal{Z}}^{(n)}_{i}>0 ~\text{for all}~ i\ge 1)> 0.  
\end{equation*}

The remaining part of the proof is essentially the same as for the maximal displacement part, and we omit the details.

\subsection{Open questions}\label{Open}

We conclude with  some open questions regarding our model and similar models. We look forward to addressing some of these in future work.

\begin{enumerate}

\item  Our result describes the leading linear term of the growth of the maximal displacement. What can be said about the second term, i.e.\ the fluctuations of $\max_{|u|=n} |X(u)| - n \cdot v_{\lambda, m}$? This is a active direction of research for branching Brownian motion and branching random walks in the multi-dimensional case. We refer to \cite{aidekon_convergence_2013} for the ``classical" result for one-dimensional branching random walks.

	\item
	To what extent can this result be generalised to other branching Markov chains? More precisely, given a Markov chain with large deviation rate function $I$, what assumptions have to be made so that the maximal displacement of the corresponding branching Markov chain (with mean offspring $m$) has a linear speed with velocity given by $\sup\{ a: I(a)\leq  \log m \}$?
	
	\item For a branching Markov chain which is statistically transitive, i.e. the underlying Markov chain is a random walk in a  homogeneous random environment, is it true in general that there is no weak recurrence phase?

	\item
	In this work we have mostly treated the large deviation principle and the rate function $I_\lambda$ as a black box. What can be said about the environment in the neighbourhood of the particles that achieve the maximum? What can be said about their ancestral path to the root? This can be compared to \cite{aidekon_speed_2014}, where it is proved that the environment seen from the $\lambda$--biased random walk is absolutely continuous with respect to $\mathtt{BGW}$, but not identical  to $\mathtt{BGW}$ itself.  
	
	\item In our model, in the critical case when $m = \frac{d_{\min}+\lambda}{2\sqrt{\lambda d_{\min}}}$, the minimal distance to the root exhibits zero velocity, despite the biased branching random walk being transient. What is the correct (sublinear) rate of escape in this scenario?
\end{enumerate}

\bigskip

\begin{paragraph}{Acknowledgments}
	We thank two anonymous referees for reading extremely carefully the first version, detecting many inaccuracies and suggesting improvements.
\end{paragraph}

\begin{paragraph}{Funding}
	This work was supported in part by the National Key R\&D Program of China (No.\ 2022YFA1006500), EPSRC grant EP/W523781/1, and National Natural Science Foundation of China (No.\ 12288201 and 12301169). 
\end{paragraph}

\bibliographystyle{abbrv} 
\bibliography{reference_v3}

\end{document}